\documentclass [10pt,a4paper]{article}
\usepackage[ansinew]{inputenc}
\usepackage{amssymb}
\usepackage{amsmath}
\usepackage{graphicx}
\usepackage{amsthm}
\usepackage{cite}
\usepackage{color}

\newcounter{marnote}

\newtheorem{theorem}{\textbf{Theorem}}[section]
\newtheorem{lemma}{\textbf{Lemma}}[section]
\newtheorem{proposition}{\textbf{Proposition}}[section]
\newtheorem{corollary}{\textbf{Corollary}}[section]
\newtheorem{remark}{\textbf{Remark}}[section]
\newtheorem{definition}{\textbf{Definition}}[section]

\allowdisplaybreaks[4]

\newcommand{\diag}{\textrm{diag}}
\def\be{\begin{equation}}
\def\ee{\end{equation}}
\def\bea{\begin{eqnarray}}
\def\eea{\end{eqnarray}}
\def\bt{\begin{theorem}}
\def\et{\end{theorem}}
\def\bl{\begin{lemma}}
\def\el{\end{lemma}}
\def\br{\begin{remark}}
\def\er{\end{remark}}
\def\bp{\begin{proposition}}
\def\ep{\end{proposition}}
\def\bc{\begin{corollary}}
\def\ec{\end{corollary}}
\def\bd{\begin{definition}}
\def\ed{\end{definition}}

\def\non{\nonumber }

\DeclareMathOperator{\tr}{tr}
\newcommand{\A}{\mathcal A}
\newcommand{\F}{\mathcal F}
\newcommand{\T}{\mathbb{T}}
\newcommand{\E}{\mathcal E}
\newcommand{\LL}{\mathcal L}
\newcommand{\Id}{\mathcal I}

\newcommand\defeq{\stackrel{\scriptscriptstyle \text{def}}=}
\newcommand{\RR}{\mathbb R}
\newcommand{\Rr}{\mathcal{R}}
\newcommand{\R}{\mathbb{R}}

\newcommand{\N}{\mathbb N}
\newcommand{\G}{\mathcal{G}}
\newcommand{\sS}{\mathcal{S}}
\newcommand{\prt}{\ensuremath{\partial}}
\newcommand{\mcA}{\mathcal{A}}
\newcommand{\mcB}{\mathcal{B}}
\newcommand{\mcC}{\mathcal{C}}
\newcommand{\veps}{\varepsilon}
\newcommand{\bE}{\overline{\mathcal{E}}}
\newcommand{\bRr}{\overline{\mathcal{R}}}
\newcommand{\bG}{\overline{\mathcal{G}}}

\newcommand{\wA}{\widetilde{\mathcal{A}}}
\newcommand{\wB}{\widetilde{\mathcal{B}}}
\newcommand{\wC}{\widetilde{\mathcal{C}}}
\newcommand{\Rx}{R^\xi}
\newcommand{\vx}{v^\xi}
\newcommand{\Ro}{R_0}

\newcommand{\dQ}{Q^{(\delta)}}
\newcommand{\dR}{R^{(\delta)}}

\newcommand{\delt}{\prt_t}

\newcommand{\eps}{\varepsilon}
\newcommand{\TT}{\mathbb{T}^2}
\newcommand{\dv}{\nabla\cdot}

\begin{document}

\title{Dynamics and flow effects in the Beris-Edwards system\\ modeling nematic liquid crystals}

\author{
  Hao Wu\footnote{School of Mathematical Sciences and Shanghai Key Laboratory for Contemporary Applied Mathematics, Fudan University,
Han Dan Road 220, Shanghai 200433, China.
  \texttt{haowufd@fudan.edu.cn, haowufd@yahoo.com}}
  \and
  Xiang Xu\footnote{
  Department of Mathematics and Statistics,
  Old Dominion University, Norfolk, VA 23529, USA.
  \texttt{x2xu@odu.edu}}
  \and
  Arghir Zarnescu\footnote{IKERBASQUE, Basque Foundation for Science, Maria Diaz de Haro 3,
48013, Bilbao, Bizkaia, Spain}\,  \footnote{BCAM,  Basque  Center  for  Applied  Mathematics,  Mazarredo  14,  E48009  Bilbao,  Bizkaia,  Spain
\texttt{ azarnescu@bcamath.org}}\, \footnote{``Simion Stoilow" Institute of the Romanian Academy, 21 Calea Grivi\c{t}ei, 010702 Bucharest, Romania}}

\date{}
\maketitle

\begin{abstract}
We consider the Beris-Edwards system modeling incompressible liquid crystal flows of nematic type. This couples a Navier-Stokes system for the fluid velocity with a parabolic reaction-convection-diffusion equation for the $Q$-tensors describing the direction of liquid crystal molecules.
In this paper, we study the effect that the flow has on the dynamics of the $Q$-tensors, by considering two fundamental aspects: the preservation of eigenvalue-range   and the dynamical emergence of defects in the limit of high Ericksen number.
\end{abstract}

\section{Introduction}
\setcounter{equation}{0}

In this paper we  consider the Beris-Ewdards system modelling nematic liquid crystals \cite{BE94}. It is one of the main PDE systems modeling nematic liquid crystals in the $Q$-tensor framework and one of the best studied mathematically \cite{ADL14,ADL15,PZ11,PZ12,GR14, GR15,W12}. It couples a Navier-Stokes system for incompressible flow with anisotropic forces and a parabolic reaction-convection-diffusion system for matrix-valued functions, i.e., the $Q$-tensors.
The Navier-Stokes system captures the fluid motion and the reaction-convection-diffusion system describes the evolution of the liquid crystal director field (see Section~\ref{sec:physical} for physical aspects).

In the Landau-de
Gennes theory  of liquid crystals (see \cite{dG93,M10,NMreview}),  the local orientation and degree of order
for neighboring liquid crystal molecules are represented by a
symmetric, traceless matrix that is an element of the {\it $Q$-tensor space}:
\be\label{def:Qtensors}
  \sS_0^{(3)}\defeq\big\{Q\in\mathbb{M}^{3\times 3}| \, Q^{ij}=Q^{ji}, \ \forall\, 1\leq i, j\leq 3, \  \tr(Q)=0
  \big\}.
\ee
The physical interpretation of the space of $Q$-tensors is presented for instance in \cite{NMreview}.

The system we are focusing on involves a number of physical constants. For the purpose of our study it is convenient to set most of these constants equal to one, keeping just a couple of specific interest to us. A specific non-dimensionalisation will be provided in Section~\ref{section:nondimensionalisation}.
In the rest of the paper, we will focus on the Beris-Ewdards system written in a non-dimensional form as:
\begin{align}
&\delt \textbf{u}+ (\textbf{u}\cdot\nabla)\textbf{u}+\nabla P\nonumber\\
&\quad =\veps\Delta  \textbf{u}-\veps^2\nabla\cdot(\nabla Q\odot \nabla Q+(\Delta Q)  Q -Q\Delta Q)\non\\
&\qquad - \veps^2\xi\nabla\cdot\bigg(\Delta  Q \big( Q+\frac{1}{3}\Id\big)+\big(Q+\frac{1}{3}\Id\big)\Delta{Q}-2\big( Q+\frac{1}{3}\Id\big)( Q:\Delta Q) \bigg)\non\\
&\qquad -2\veps\xi\kappa\nabla\cdot\bigg(\big(Q+\frac{1}{3}\Id\big)\big(a\tr(Q^2)-b\tr(Q^3)+c\tr^2(Q^2)\big)\bigg)\non\\
&\qquad +2\veps\xi\kappa\nabla\cdot\bigg(\big(Q+\frac{1}{3}\Id\big)\big(a Q-b \big[Q^2-\frac{1}{3}\tr(Q^2)\Id\big]+c Q\tr(Q^2)\big)\bigg),\label{nondim:u1}\\
&\nabla\cdot \textbf{u}=0,\label{nondim:du}\\
&\delt Q+ \textbf{u}\cdot\nabla Q-S(\nabla \textbf{u}, Q)=\veps \Delta Q-\left( aQ-b\big[Q^2-\frac{1}{3}\tr(Q^2)\Id\big]+ cQ\tr(Q^2)\right),
\label{nondim:Q1}
\end{align}
where
\begin{align}
&S(\nabla \textbf{u}, Q) = (\xi D+ \Omega)\big( Q+\frac13 \Id
\big)+\big(Q+\frac13 \Id \big)(\xi
D- \Omega) -2\xi\big(Q+\frac13 \Id \big)\tr(Q\nabla \textbf{u}), \quad \text{for some} \ \xi\in \mathbb{R}.\non
\end{align}
Here, we denote by $\TT=[-\pi,\pi]^2$ the two dimensional torus. Then $\textbf{u}(x, t): \TT\times (0, +\infty) \rightarrow \R^3$
represents the fluid velocity field and $Q(x, t):
\TT\times (0, +\infty) \rightarrow \sS_0^{(3)}$ stands for the liquid crystal directors, with $D=\dfrac{\nabla\mathbf{u}+\nabla^T\mathbf{u}}{2}$,
$\Omega=\dfrac{\nabla\mathbf{u}-\nabla^T\mathbf{u}}{2}$ being the
symmetric and skew-symmetric parts of the rate of strain tensor,
respectively. The constant $\xi$ in $S(\nabla \textbf{u}, Q)$ is a measure of the ratio between the tumbling and the aligning effects
that a shear flow exerts on the liquid crystal director field.

There has been a vast recent literature on the study of
well-posedness as well as long-time dynamics of the Beris-Edward
system \eqref{nondim:u1}--\eqref{nondim:Q1}.
We refer interested readers to \cite{ADL15, D15, DFRSS14,
GR14, GR15, PZ12} for the discussion related to the simplified system with $\xi=0$,
and \cite{ADL14, CRWX15, DAZ, PZ11} for the full system with a non-vanishing $\xi$.

The system \eqref{nondim:u1}--\eqref{nondim:Q1} contains a
significant number of terms and generates considerable analytical
difficulties, mostly due to the presence of the Navier-Stokes part that describes the effects of the fluid.
Often in the physical
literature it is assumed that the fluid can be neglected and one
works with some suitably simplified fluid-free versions of the system.
However, it is not a priori clear what is lost through these simplifications and
the main aim of this work is to study two interlinked fundamental
issues:

\begin{itemize}
\item The preservation of eigenvalue-range,
\item The partial decoupling in the limit of high Ericksen number.
\end{itemize}

We aim to provide an understanding of these qualitative issues. Thus, in this paper we will focus on the simplest possible setting, from a technical point of view, that allows to bypass regularity issues (in particular, for the Navier-Stokes system). Namely, we will work on the two dimensional torus $\TT$ and with sufficiently smooth solutions, i.e., the global strong solutions.
The existence of global strong solutions in this setting for the full system \eqref{nondim:u1}--\eqref{nondim:Q1} was done in \cite{CRWX15}, however assuming that the range of $Q$-tensors is the set of two dimensional tensors, i.e. $2\times 2$ matrices.
The framework we use here is slightly different such that we work essentially in a three dimensional setting for the target spaces (cf. \cite[Remark 1.1]{PZ11}). More precisely, we consider the fluid velocity $\mathbf{u}$ and the $Q$-tensor defined on $\TT$ but take values in the spaces $\mathbb{R}^3$ and $\sS_0^{(3)}$ (see \eqref{def:Qtensors}), respectively.
This setting can be easily obtained from the work in \cite{CRWX15} by assuming that we work in three dimensional space with initial data independent of one spatial direction, so essentially a two dimensional datum (in the domain). This property of the initial data of being two dimensional is preserved by the flow and all the arguments in \cite{CRWX15} follow, because the only thing that matters for obtaining global strong solutions is the Sobolev embedding theorem that continues to hold as the domain is kept to be essentially two dimensional.
Hence, we shall simply assume the existence of global strong solutions defined as follows:

\begin{definition}
Consider $(\mathbf{u}_0,Q_0)\in H^1(\TT;\R^3)\times H^2(\TT;\sS_0^{(3)})$  with $\dv \mathbf{u}_0=0$.
A pair $(\mathbf{u},Q)$ is called a global strong solution of the Beris-Edward system \eqref{nondim:u1}--\eqref{nondim:Q1} with initial data $(\mathbf{u}_0,Q_0)$,
if
\begin{equation}
\begin{aligned}
&\mathbf{u}\in C([0,\infty);H^1(\TT;\R^3))\cap L^2_{loc}(0,\infty; H^2(\TT;\R^3)),\\
&Q\in C([0,\infty);H^2(\TT;\sS^{(3)}_0))\cap L^2_{loc}(0,\infty; H^3(\TT;\sS^{(3)}_0)),
\end{aligned}\non
\end{equation}
and $(\mathbf{u},Q)$ satisfies the system \eqref{nondim:u1}--\eqref{nondim:Q1} a.e. in $\TT\times[0,\infty)$.
\end{definition}

The first issue of interest to us, i.e., the preservation
of eigenvalue-range, concerns the behaviour of eigenvalues
of $Q$-tensors under the fluid dynamics. The main question is to
understand what will happen as time evolves if the eigenvalues of the
initial datum $Q_0(x)$ are in a convex set (with $x\in\TT$ in the
current stage): will they stay in the same set or not? As explained
in the next section this is a fundamental issue motivated by the
physical interpretation of the $Q$-tensors.

 It is already known (see for instance, \cite[Theorem 3]{GR15}) that
 if one takes $\xi=0$ in the system \eqref{nondim:u1}--\eqref{nondim:Q1}, then a maximum principle
is valid for the $Q$-equation, i.e., the $L^\infty$-norm of the initial datum $Q_0$ will be preserved for the solution
$Q(x, t)$ during the evolution. However, as pointed out in \cite{IXZ14} the preservation of eigenvalue-range is a much more subtle issue than the preservation
of $L^\infty$-norm of $Q_0$.

On the other hand, if the fluid is neglected, i.e., $\mathbf{u}=0$, then the $Q$-equation \eqref{nondim:Q1} is simply
reduced to a gradient flow of the free energy $\F(Q)$ (see \eqref{free energy} for its definition). It is proved
in \cite{IXZ14} (in the whole space setting) that this gradient flow will preserve the convex
hull of eigenvalues of $Q_0$ for any regular solution $Q(x,t)$. The proof therein is based on an operator splitting idea
and a nonlinear Trotter product formula \cite{T97}. More precisely,
the gradient flow is ``splitted" into a heat flow and an ODE system so that
the initial eigenvalue constraints are shown to be preserved by both
sub-flows. Then the Trotter formula performs the
combination. Returning to our full system \eqref{nondim:u1}--\eqref{nondim:Q1}, we note that  the
structure of the current $Q$-equation is nevertheless much more
complicated such that the argument for the heat flow part in
\cite{IXZ14} can no longer be applied. To overcome the corresponding mathematical
difficulty, we shall introduce a singular potential discussed in
\cite{BM10} (see \cite{EKT16, FRSZ14, FRSZ15, W12} for various
applications), and make a thorough exploration of its special
properties.

To this end, we will first investigate the special case $\xi=0$.
In the literature this
is often referred to as the {\it ``co-rotational case"}. Taking
$\xi=0$ generates considerably simpler equations and in particular,
ensures the validity of a maximum principle for $Q$. In this special case, we will
show that in fact one even has the preservation of a suitable convex
eigenvalue-range and thus extending the
results in \cite{IXZ14} to the current case with flow. As mentioned in \cite{IXZ14},
the preservation of $L^\infty$-norm of $Q$ (i.e., the maximum principle) can be viewed as a
``weaker" version of the preservation of physicality (i.e., the preservation of eigenvalues).
\begin{theorem}\label{main theorem} {\bf [Eigenvalue preservation in the co-rotational case $\xi=0$].}
Let $T>0$, $\xi=0$, $\kappa>0$, $a\in \mathbb{R}$, $b>0$, $c>0$ and
$(\mathbf{u}_0,Q_0)\in H^1(\TT;\R^3)\times H^2(\TT;\sS_0^{(3)})$
with $\dv \mathbf{u}_0=0$. We assume that the eigenvalues of the initial datum $Q_0$
satisfy
\begin{equation}\label{initial setting}
\mathrm{\lambda}_i(Q_0(x))\in\left[-\dfrac{b+\sqrt{b^2-24ac}}{12c},\ \ \dfrac{b+\sqrt{b^2-24ac}}{6c}\right],
\qquad \forall\, x\in\TT,\  1\leq i\leq 3.
\end{equation}
Furthermore, we impose the following restriction on the coefficients:
\begin{equation}\label{restriction:a}
|a|<\frac{b^2}{3c}.
\end{equation}
Let $(\mathbf{u}, Q)$ be the global strong solution
to the system  \eqref{nondim:u1}--\eqref{nondim:Q1} on $[0,T]$ with initial data $(\mathbf{u}_0, Q_0)$.
Then for any $t\in [0,T]$ and $x\in\TT$, the eigenvalues of
$Q(x,t)$ stay in the same interval as in \eqref{initial setting}.
\end{theorem}

 We infer from Theorem \ref{main theorem} that for our system
\eqref{nondim:u1}--\eqref{nondim:Q1} the fluid velocity field $\mathbf{u}$ will \emph{not} affect
the initial eigenvalue constraint on $Q$ as time evolves, provided that $\xi=0$.

On the other hand, generally one cannot prove any maximum principle of
$|Q|$ for the full system \eqref{nondim:u1}--\eqref{nondim:Q1} unless $\xi=0$.
 Next, we will show that for the general case with $\xi\not=0$, the preservation of eigenvalues indeed does not hold as well.
 For this purpose, we will argue by contradiction and use a simplified version of the full system \eqref{nondim:u1}--\eqref{nondim:Q1}. This will be obtained as the so-called high Ericksen number limit, which corresponds to the formal limit $\veps\to 0$ in our setting.
 The resulting system we obtain is only weakly coupled, namely, it is a system coupling an Euler equation for the fluid velocity with a reaction-convection equation for the order parameter $Q$.
 First, we prove

\begin{theorem}\label{thm:highEricksen}
{\bf [High Ericksen number limit in the co-rotational case].}
Let $\mathbf{u}_0\in H^{5}(\TT; \mathbb{R}^3)$  with $\dv \mathbf{u}_0=0$ and $Q_0\in H^{4}(\TT;\sS_0^{(3)})$.
Consider the {\it ``limit  system"}:

\begin{align}
&\delt \mathbf{v}+ \mathbf{v}\cdot\nabla\mathbf{v}+\nabla q=0,\label{nondim:v10}\\
&\nabla\cdot \mathbf{v}=0,\label{nondim:dv}\\
&\delt R+\mathbf{v}\cdot\nabla R-\Omega_{\mathbf{v}}R+R\Omega_{\mathbf{v}}= -aR+b\big[R^2-\frac{1}{3}\tr(R^2)\Id\big]- cR\tr(R^2),
\label{nondimn:R10}
\end{align}
 with $\Omega_\mathbf{v}:=\displaystyle{\frac{\nabla \mathbf{v}-\nabla^T \mathbf{v}}{2}}$.
Then for all $\eps\in(0,1)$ and any strong solution  $(\mathbf{u}^\eps,Q^\eps)$ starting from initial data $(\mathbf{u}_0,Q_0)$ of the system \eqref{nondim:u1}--\eqref{nondim:Q1} with $\xi=0$,
if we denote $\mathbf{w}^\eps=\mathbf{u}^\eps-\mathbf{v}$,   $S^\eps=Q^\eps-R$, then for any time $T>0$ we have:
\begin{align}
&\|\mathbf{w}^\eps\|_{L^2}^2+\veps^2\|\nabla S^\eps\|_{L^2}^2 +\veps\|S^\eps\|_{L^2}^2 \leq C_T \veps^2, \quad \forall\, t\in [0,T],\non\\
& \int_0^T\left(\veps \|\nabla \mathbf{w}^\eps\|_{L^2}^2 +\veps^3\|\Delta S^\eps\|_{L^2}^2+ \veps^2\|\nabla S^\eps\|_{L^2}^2\right)dt\leq C_T \veps^2.\non
\end{align}
\end{theorem}

  We note that the assumption $\xi=0$ in the above theorem can be removed, provided that one has an a priori $L^\infty$-bound of $|Q|$ that is also uniform in $\veps>0$ . On the other hand, if we assume that the property of eigenvalue preservation holds, then we actually have such an a priori $L^\infty$-bound. As a consequence, we can obtain the limit system \eqref{nondim:v10}--\eqref{nondimn:R10} in the non co-rotational case i.e.,  $\xi\neq 0$, for which we can show that in general one cannot except the preservation of the initial eigenvalue range. This is the overall basic strategy leading to our next result:

\begin{theorem}\label{thm:lackofeigenv}{\bf [Lack of eigenvalue preservation in the non co-rotational case].}

There exist some $\xi\not=0$, $\veps>0$, $\kappa>0$, $a$, $b$, $c\in \R$ satisfying $b>0$, $c>0$, $|a|<\displaystyle{\frac{b^2}{3c}}$, some initial data  $\mathbf{u}_0\in H^{5}(\TT;\R^3)$ with $\dv \mathbf{u}_0=0$ and  $Q_0\in H^{4}(\TT;\sS^{(3)}_0)$ that satisfies \eqref{initial setting}, such that for some $(x,t)\in \TT \times \mathbb{R}^+$
$$
\mathrm{\lambda}_i(Q(x,t))\notin\left[-\dfrac{b+\sqrt{b^2-24ac}}{12c},\dfrac{b+\sqrt{b^2-24ac}}{6c}\right],
$$
where $(\mathbf{u},Q)$ is the global strong solution to the full system  \eqref{nondim:u1}--\eqref{nondim:Q1} with initial data $(\mathbf{u}_0, Q_0)$ and parameters $\eps$ and $\xi$.
\end{theorem}

Finally, in Section~\ref{sec:defectsdyn} we return to the system obtained in the limit of high Ericksen number with $\xi=0$, i.e., \eqref{nondim:v10-}--\eqref{nondimn:R10-}. As argued in the physical motivation in Section~\ref{sec:physical}, this is relevant for understanding certain patterns, specific to liquid crystals, the {\it ``defect patterns"} (see
\cite{dG93}). We consider certain specific solutions of the Euler equations and provide examples of flows that are non-singular but are still capable of generating various types of such defect patterns. More precisely, we present two distinct way of generating defects: the phase mismatch in absence of the flow and these vorticity-driven defects.

The rest of this paper is organized as follows. In the next section
we provide a couple of background details concerning the  physical
relevance of our study. This is a section that the
mathematically-minded readers can safely skip. Then in
Section~\ref{sec:eigenpres}   we provide the proof of
Theorem~\ref{main theorem} concerning the preservation of
eigenvalue-range in the co-rotational case with $\xi=0$.
In Section~\ref{sec:highEricksen}, we obtain
the limit of high Ericksen number and prove the error estimates, i.e.,   Theorem~\ref{thm:highEricksen}. Then in
Section~\ref{sec:lacknoncorot},
we can use the reduced system to argue for the initial system that for
$\xi\not=0$ we cannot always have eigenvalue preservation, thus
proving Theorem~\ref{thm:lackofeigenv}.
Finally, in Section~\ref{sec:defectsdyn} we return to the reduced system with $\xi=0$ and
study its implications concerning the dynamical appearance of defect
patterns. In the Appendix, we recall some rather technical results
that are necessary in Sections~\ref{sec:eigenpres}
and~\ref{sec:defectsdyn} but that seem difficult to pinpoint in the literature.

\subsection*{Notational Conventions}
Throughout the paper, we assume the Einstein summation convention
over repeated indices. We define $A: B \defeq \tr(A^T B)$ for matrices $A,
B\in\mathbb{M}^{3\times 3}$, and denote by $|Q|$ the Frobenius norm
of the matrix $Q\in\mathbb{M}^{3\times 3}$ such that $|Q|=\sqrt{Q:Q}$.
Besides,  the $3\times 3$ identity matrix will be denoted by $\Id$.
The partial derivative with
respect to spatial variable $x_k$ of the $(i,j)$-component of $Q$ is denoted by $\partial_k
Q^{ij}$ or $Q_{ij,k}$. Next, we define the matrix valued $L^p$ space ($1 \leq p \leq \infty$) by
\begin{equation*}
  L^p(\TT; \sS_0^{(3)} ) \defeq \Big\{Q:\TT\rightarrow\sS_0^{(3)},\ |Q|\in L^p(\TT;\RR)\Big\}.
\end{equation*}
For the sake of simplicity, unless explicitly pointed out, we shall
simply use $L^p(\TT)$ to express $L^p(\TT; \R^{3})$ for the fluid velocity and $L^p(\TT; \sS_0^{(3)})$ for the $Q$-tensors. In a similar manner, we use $H^{k}(\TT)$ ($k\in\mathbb{N}$) to represent the Sobolev spaces $H^k(\TT; \R^{3})$ for the fluid velocity and $H^k(\TT; \sS_0^{(3)})$
for $Q$-tensors, respectively.
We will at times abbreviate
$\|\cdot\|_{L^p(\TT)}$ as $\|\cdot\|_{L^p}$,
$\|\cdot\|_{H^k(\TT)}$ as $\|\cdot\|_{H^k}$ etc.

\section{Physical aspects}
\setcounter{equation}{0}
\label{sec:physical}
\subsection{Eigenvalue-range constraints}

The main characteristic of nematic liquid crystals is the locally preferred
orientation of the nematic molecule directors. This can be  described  by the
$Q$-tensors, that are suitably normalized second order moments of the probability
distribution function of the molecules. More precisely, if $\mu_x$ is a probability
measure on the unit sphere $\mathbb{S}^2$, representing the orientation of liquid crystal  molecules at a point
$x$ in space, then a $Q$-tensor denoted by ${Q}(x)$ is a symmetric and traceless $3\times 3$ matrix defined as
\begin{equation}\nonumber
  Q(x)=\int_{\mathbb{S}^2}\left(\mathbf{p}\otimes  \mathbf{p}-\frac{1}{3}\Id\right)\,d\mu_x(\mathbf{p}).
\end{equation}
It is supposed to be a  crude measure of how the probability measure $\mu_x$ deviates
from the isotropic measure $\bar\mu$ where $d\bar\mu=\frac{1}{4\pi}dA$, see \cite{NMreview}.
The fact that $\mu_x$ is a probability measure imposes a constraint on the eigenvalues $\{\lambda_i(Q)\}$ of $Q$ such that (see \cite{NMreview})
$$ -\frac{1}{3}\leq \lambda_i(Q)\leq  \frac{2}{3},\quad \sum_{i=1}^3\lambda_i(Q) =0.$$
Thus, not any symmetric and traceless $3\times 3$ matrix is a {\it physical} $Q$-tensor
but only those whose eigenvalues  take values  in $(-\frac{1}{3},\frac{2}{3})$.

\subsection{Non-dimensionalisation and the Ericksen number}
\label{section:nondimensionalisation}
The free energy of liquid crystal molecules is given by
\begin{equation}\label{free energy}
\F(Q)\defeq
\int_{\TT}\frac{L}{2}|\nabla{Q}|^2\underbrace{+\frac{\tilde a}{2}\tr(Q^2)-\frac{\tilde b}{3}\mathrm{tr}(Q^3)+\frac{\tilde c}{4}\tr^2(Q^2)}_{\defeq f_B(Q)}\,dx,
\end{equation}
where $L$, $\tilde a$,  $\tilde b$,  $\tilde c \in \mathbb{R}$ are material dependent and temperature-dependent constants that satisfy \cite{M10, MZ10}
\begin{equation}\label{AssumptionsBulk}
L>0, \  \tilde b\geq 0, \  \tilde c>0.
\end{equation}
Here, we use the one constant approximation of the Oseen-Frank
energy for the sake of simplicity (see \cite{BM10}).

The  Beris-Edwards system we are going to consider is given by
\begin{align}
&\rho(\mathbf{u}_t+\mathbf{u}\cdot\nabla\mathbf{u})-\nu\Delta\mathbf{u}+\nabla{P}=\nabla\cdot(\tau+\sigma),\label{NSE-0}\\
&\nabla\cdot\mathbf{u}=0,\label{incomp-0}\\
&Q_t+\mathbf{u}\cdot\nabla{Q}-S(\nabla\mathbf{u}, Q)=\Gamma
H(Q).\label{Q equ-0}
\end{align}
Here $\mathbf{u}(x, t): \TT\times (0, +\infty) \rightarrow \R^3$
represents the fluid velocity field and $Q(x, t):
\TT\times (0, +\infty) \rightarrow \sS_0^{(3)}$ stands for the
$Q$-tensor of liquid crystal molecules.
The positive constants
$\nu$ and $\Gamma$ denote the fluid viscosity and the
macroscopic elastic relaxation time for the molecular orientation
field, respectively.

In the $Q$-equation \eqref{Q equ-0}, the tensor $H=H(Q)$ is
defined to be the minus variational derivative of the free energy $\F(Q)$
with respect to $Q$ under symmetry and tracelessness constraints:
\begin{align}
H(Q)&\defeq -\dfrac{\partial \F(Q)}{\partial Q}=L\Delta{Q}-\frac{\partial f_B}{\partial Q}
\nonumber\\
&=L\Delta Q-\tilde aQ+\tilde b\Big(Q^2-\frac13\tr(Q^2)\Id\Big)-\tilde cQ\tr(Q^2),\label{H}
\end{align}
meanwhile the term $S(\nabla\mathbf{u}, Q)$ in \eqref{Q equ-0} reads as follows
\begin{equation}\label{S1}
S(\nabla\mathbf{u}, Q)\defeq (\xi D+\Omega)\big(Q+\frac13 \Id
\big)+\big(Q+\frac13 \Id \big)(\xi
D-\Omega)-2\xi\big(Q+\frac13 \Id \big)\tr(Q\nabla\mathbf{u}),
\end{equation}
where $D=\dfrac{\nabla\mathbf{u}+\nabla^T\mathbf{u}}{2}$ and
$\Omega=\dfrac{\nabla\mathbf{u}-\nabla^T\mathbf{u}}{2}$ stand for the stretch and the vorticity tensor, respectively. It is worth pointing out that the constant $\xi$ in \eqref{S1} is a measure of the ratio between the tumbling and the aligning effects
that a shear flow exerts on the liquid crystal director field.

Next, the anisotropic forcing terms in the Navier-Stokes system \eqref{NSE-0} are
elastic stresses caused by the presence of liquid crystal molecules,
which include the symmetric part
\begin{equation}
\tau\defeq-\xi\big(Q+\frac{1}{3}\Id\big)H-\xi{H}\big(Q+\frac13\Id\big)+2\xi\big(Q+\frac13\Id\big)\tr(QH)
-L\nabla{Q}\odot\nabla{Q},
\end{equation}
and the skew-symmetric part
\begin{equation}
\sigma\defeq QH-HQ.
\end{equation}

Thus, the coupled system \eqref{NSE-0}--\eqref{Q equ-0} takes the following explicit form:
\begin{align}
&\rho(\mathbf{u}_t+\mathbf{u}\cdot\nabla\mathbf{u})-\nu\Delta\mathbf{u}+\nabla{P}\non\\
&\quad = L\nabla\cdot(Q\Delta{Q}-(\Delta{Q})Q)-L\nabla\cdot(\nabla{Q}\odot\nabla{Q})\non\\
&\qquad -\xi L\nabla\cdot\bigg(\Delta Q(Q+\frac{1}{3}\Id)+(Q+\frac{1}{3}\Id)\Delta Q-2(Q+\frac{1}{3}\Id)(Q:\Delta Q) \bigg)\non\\
&\qquad -2\xi\nabla\cdot\bigg(\big(Q+\frac{1}{3}\Id\big)(Q:\frac{\partial f_B}{\partial Q}-\frac{\partial f_B}{\partial Q})\bigg)
\label{NSE}\\
&\nabla\cdot\mathbf{u}=0,
\label{incomp}\\
&Q_t+\mathbf{u}\cdot\nabla{Q}=(\xi D+\Omega)\big(Q+\frac13 \Id
\big)+\big(Q+\frac13 \Id \big)(\xi
D-\Omega)-2\xi\big(Q+\frac13 \Id \big)\tr(Q\nabla\mathbf{u})\non\\
&\qquad \qquad \qquad \ \ +\Gamma\left(L\Delta{Q}- \tilde aQ+ \tilde b\Big[Q^2-\frac{1}{3}\tr(Q^2)\Id\Big]-\tilde cQ\tr(Q^2)\right).
\label{Q-equ}
\end{align}
Besides, we assume the system \eqref{NSE}--\eqref{Q-equ} is subject to the initial conditions
\begin{equation}\label{IC-BC}
\mathbf{u}(x, 0)=\mathbf{u}_0(x) \;\mbox{with} \;
\nabla\cdot\mathbf{u}_0=0, \quad Q(x, 0)=Q_0(x)\in \sS_0^{(3)}, \quad \forall\, x\in \TT.
\end{equation}

The full system \eqref{NSE}--\eqref{IC-BC} involves several dimensional parameters (see \cite{Virgaelectric, MkadGart} and the references therein):

 \begin{itemize}
 \item $\rho$ is the density (assumed to be constant here), expressed in units of $Kg\cdot m^{-3}$.
 \item $L$ is an elastic constant, measuring spatial distortions of the order parameter $Q$, expressed in units of $Kg\cdot m\cdot s^{-2}$,
 \item  $\tilde a$, $\tilde b$ and $\tilde c$ are material-dependent parameters (at a fixed temperature), expressed in units of $J\cdot m^{-1}=Kg\cdot s^{-2}\cdot m^{-1}$,
 \item $\nu$ is the fluid viscosity, expressed in units of  $ Kg\cdot m^{-1}\cdot s^{-1}$,
 \item $\Gamma$ is the inverse of a viscosity coefficient, expressed in units of $m\cdot s\cdot Kg^{-1}$.
 \end{itemize}
We remark that the parameter $\xi$ is a non-dimensional constant measuring a characteristic of the molecules.

Following the non-dimensionalisation performed in \cite{LiuCarme}, we denote the characteristic length, velocity and density by $\hat l$, $\hat u$, $\hat\rho$, respectively. Then we define the non-dimensional distance $\bar x$ and time $\bar t$ as:
$$
x=\hat l\bar x,\quad  t=\frac{\hat l}{\hat u}\bar t.
$$
 Take the non-dimensional variables $\bar \rho$, $\bar{\mathbf{u}}$, $\bar Q$ and $\bar P$ to be such that:
$$ \rho=\hat{\rho} \bar \rho,\quad  \mathbf{u}(x,t)=\hat u\bar{\mathbf{u}}(\bar x,\bar t),\ \ Q(x,t)=\bar Q(\bar x,\bar t), \ \ P(x,t)=\hat\rho \hat u^2\bar P(\bar x,\bar t).$$
Inserting the above relations into the system \eqref{NSE}--\eqref{Q-equ}, we obtain
\begin{align}
&\bar \rho(\partial_{\bar t} \bar{\mathbf{u}}+\bar{\mathbf{u}}\cdot\nabla_{\bar x} \bar{\mathbf{u}})+\nabla_{\bar x} \bar P \nonumber\\
&\quad =\frac{\nu}{\hat\rho\hat l\hat u} \Delta_{\bar x} \bar{\mathbf{u}}
-\frac{L}{\hat\rho \hat l^2\hat u^2}\nabla_{\bar x}\cdot[\nabla_{\bar x}\bar Q\odot \nabla_{\bar x}\bar Q+(\Delta_{\bar x}\bar Q) \bar Q-\bar Q\Delta_{\bar x}\bar Q]\non\\
&\qquad
- \xi \frac{L}{\hat\rho \hat l^2\hat u^2}\nabla_{\bar x}\cdot\bigg(\Delta_{\bar x} \bar Q(\bar Q+\frac{1}{3}\Id)+(\bar Q+\frac{1}{3}\Id)\Delta_{\bar x} \bar Q-2(\bar Q+\frac{1}{3}\Id)(\bar Q:\Delta_{\bar x} \bar Q) \bigg)\non\\
&\qquad
-2\xi\nabla_{\bar x}\cdot\bigg(\big(\bar Q+\frac{1}{3}\Id\big)\big(\frac{\tilde a}{\hat \rho\hat u^2}\mathrm{tr}(\bar Q^2)-\frac{\tilde b}{\hat \rho\hat u^2}\textrm{tr}(\bar Q^3)+\frac{\tilde c}{\hat \rho\hat u^2}\textrm{tr}^2(\bar Q^2)\big)\bigg)\non\\
&\qquad
+ 2\xi\nabla_{\bar x}\cdot\bigg(\big(\bar Q+\frac{1}{3}\Id\big)\big(\frac{\tilde a}{\hat \rho\hat u^2}\bar Q - \frac{\tilde  b}{\hat \rho\hat u^2}(\bar Q^2-\frac{1}{3}\mathrm{tr} (\bar Q^2)\Id)+\frac{\tilde  c}{\hat \rho\hat u^2}\bar Q\mathrm{tr}(\bar Q^2)\big)\bigg),
\label{system:non dim-}\\
&\nabla_{\bar x}\cdot\bar{\mathbf{u}}=0,\\
&\partial_{\bar t} \bar Q +\bar{\mathbf{u}}\cdot\nabla_{\bar x} \bar Q
=(\xi \bar D+\bar \Omega)\big(\bar Q+\frac13 \Id
\big)+\big(\bar Q+\frac13 \Id \big)(\xi
\bar D-\bar \Omega) -2\xi\big(\bar Q+\frac13 \Id \big)\tr(\bar Q\nabla_{\bar x}\bar{\mathbf{u}})\non\\
&\qquad \qquad \qquad \ \ \ +\frac{\Gamma \hat l}{\hat u}\left(\frac{L}{\hat l^2}\Delta_{\bar x} \bar  Q-\tilde a\bar Q
         +\tilde b\big(\bar Q^2-\frac{1}{3}\mathrm{tr}(\bar Q^2)\Id \big)
         - \tilde c\bar Q\mathrm{tr}(\bar Q)^2\right),
\label{system:non dim-2}
\end{align}
where $\bar D=\dfrac{\nabla_{\bar x}\bar{\mathbf{u}}+\nabla_{\bar x}^T\bar{\mathbf{u}}}{2}$ and
$\bar \Omega=\dfrac{\nabla_{\bar x}\bar{\mathbf{u}}-\nabla_{\bar x}^T\bar{\mathbf{u}}}{2}$.

On the other hand, we recall that the non-dimensional Ericksen and Reynolds numbers can be defined as follows (see \cite{LiuCarme}):

\be
\label{def:E}
\mathcal{E}:=\frac{\nu \hat l \hat u}{L},
\ee
\be
\label{def:R}
\mathcal{R}:=\frac{\hat \rho \hat l \hat u}{\nu}.
\ee
Furthermore, we introduce the following non-dimensional quantities:
\begin{equation}
\begin{aligned}
&\wA:=\frac{1}{\hat \rho\hat u^2}\tilde a,\quad \wB:=\frac{1}{\hat \rho\hat u^2}\tilde b,\quad\wC:=\frac{1}{\hat \rho\hat u^2}\tilde c,\non\\
&\mathcal{A}:=\frac{\Gamma \hat l}{\hat u}\tilde a,\quad \mathcal{B}:=\frac{\Gamma \hat l}{\hat u}\tilde b,\quad \mathcal{C}:=\frac{\Gamma \hat l}{\hat u}\tilde c,\non\\
&\mathcal{G}:=\frac{\Gamma L}{\hat u\hat l},
\end{aligned}
\end{equation}
which implies that the system \eqref{system:non dim-}--\eqref{system:non dim-2} becomes, in a non-dimensional form:
\begin{align}
& \bar \rho(\partial_{\bar t} \bar{\mathbf{u}}+\bar{\mathbf{u}}\cdot\nabla_{\bar x}\bar{\mathbf{u}})+\nabla_{\bar x}\bar P\non\\
&\quad =\frac{1}{\Rr}\Delta_{\bar x} \bar{\mathbf{u}}-\frac{1}{\Rr\E}\nabla_{\bar x}\cdot[\nabla_{\bar x}\bar Q\odot \nabla_{\bar x}\bar Q+(\Delta_{\bar x}\bar Q) \bar Q-\bar Q\Delta_{\bar x}\bar Q]\non\\
&\qquad - \frac{\xi}{\Rr\E}\nabla_{\bar x}\cdot\bigg(\Delta_{\bar x} \bar Q(\bar Q+\frac{1}{3}\Id)+( \bar Q+\frac{1}{3}\Id)\Delta_{\bar x} \bar Q - 2(\bar Q+\frac{1}{3}\Id)(\bar Q:\Delta_{\bar x} \bar Q) \bigg)\non\\
&\qquad -2\xi\nabla_{\bar x}\cdot\bigg(\big(\bar Q+\frac{1}{3}\Id\big)\big(\wA\textrm{tr}(\bar Q^2)-\wB\textrm{tr}(\bar Q^3)+\wC\textrm{tr}^2(\bar Q^2)\big)\bigg)\non\\
&\qquad +2\xi\nabla_{\bar x}\cdot\bigg(\big(\bar Q+\frac{1}{3}\Id\big)\big(\wA \bar Q-\wB (\bar Q^2-\frac{1}{3}\mathrm{tr}(\bar Q^2)\Id)+ \wC \bar Q\mathrm{tr}(\bar Q^2)\big)\bigg),
\label{system:non dim1}\\
&\nabla_{\bar x}\cdot\bar{\mathbf{u}}=0,\\
&\partial_{\bar t} \bar Q+\bar{\mathbf{u}}\cdot\nabla_{\bar x} \bar Q
=(\xi\bar D+\bar \Omega)\big(\bar Q+\frac13 \Id
\big)+\big(\bar Q+\frac13 \Id \big)(\xi
\bar D-\bar \Omega)-2\xi\big(\bar Q+\frac13 \Id \big)\tr(\bar Q\nabla_{\bar x}\bar{\mathbf{u}})\non\\
&\qquad \qquad \qquad \qquad
+\G \Delta_{\bar x} \bar  Q-\mcA\bar Q+\mcB\big(\bar Q^2-\frac{1}{3}\mathrm{tr}(\bar Q^2)\Id\big)-\mcC\bar Q\mathrm{tr}(\bar Q^2).
\end{align}
\begin{remark}
Taking into account the definitions \eqref{def:E}, \eqref{def:R} for the Ericksen and Reynolds numbers we see that the physical parameters $\rho$, $\nu$ are material-dependent, hence fixed, and we can only vary $\hat l$ or $\hat u$. We observe that $\hat l$ is essentially a unit of the length measurement, which can be changed for instance from $10^{-6}$ meters to $10^6$ meters. Hence, formally increasing $\hat l$ amounts to obtaining a system relevant on a much larger scale. Similarly, fixing $\hat l$ and changing $\hat u$ amounts to a change of scale in time. The most important thing to notice is that a change of scale implies a simultaneous change in all the non-dimensional parameters.
\end{remark}

In order to study defect patterns, which appear as localized high gradients, it is useful to take a {\it large space scale $\hat l\to\infty$} for fixed time scale $\hat t$ (which also implies that $\hat u\to\infty$ and $\displaystyle{\frac{\hat l}{\hat u}}$ is fixed).  This corresponds to taking {\it large Ericksen and Reynolds numbers, simultaneously}:
\begin{equation}
\left\{
\begin{aligned}
&\E=\frac{\bE}{\veps},\ \Rr=\frac{\bRr}{\veps},\\
& \wA=\veps\wA_0,\  \wB=\veps\wB_0, \ \wC=\veps\wC_0,\\
& \mcA=\mcA_0,\ \mcB=\mcB_0,\ \mcC=\mcC_0,\\
& \G=\veps\bG,
\end{aligned}
\right.
\label{rel:limit2}
\end{equation}
where $\varepsilon>0$ is a parameter that will be sent to $0$.

In the remaining part of the paper, we focus on the scaling \eqref{rel:limit2} and for notational simplicity, we set
$$
\bE=\bRr=\bG=1,\quad a:=\mcA_0,\ b:=\mcB_0,\ c:=\mcC_0,\quad \kappa:=\frac{\wA_0}{\mcA_0}=\frac{\wB_0}{\mcB_0}=\frac{\wC_0}{\mcC_0}.
$$
Dropping the bars of unknown variables and the spatial/temporal variables, setting $\bar \rho=1$ without loss of generality, the system we are going to study becomes the one in the introduction, namely \eqref{nondim:u1}--\eqref{nondim:Q1}.

\section{Eigenvalue-range preservation in the co-rotational case}
\label{sec:eigenpres}
\setcounter{equation}{0}

In this section, we aim to prove Theorem~\ref{main theorem}.
To this end, we
consider the system \eqref{nondim:u1}--\eqref{nondim:Q1} with
$\xi=0$. Assume that $(\mathbf{u}, Q)$ is a strong solution to this
system, existing on some time interval $[0,T]$. Below we will
investigate the $Q$-equation \eqref{nondim:Q1} only, regarding the
velocity $\mathbf{u}$ (and thus $\Omega=\dfrac{\nabla \mathbf{u}-\nabla^T\mathbf{u}}{2}$) in the convection terms as
given and sufficiently smooth functions. In Subsection~\ref{subsec:proofofweakflow}, we will show that this smoothness assumption can be relaxed in certain sense.
Then the proof of Theorem \ref{main theorem}
will be obtained in the last subsection of this section, by
using a nonlinear Trotter product formula whose mechanism is
explained in an abstract setting in the Appendix, Section~\ref{sec:Trotter}  (see \cite{T97} for the autonomous case).

\subsection{Eigenvalue-range preservation for decomposed systems}

To begin with, we denote by $S(t, \bar{S})$, or
$S(t, \cdot)\bar{S}\in \sS_0^{(3)}$ the flow generated by the ODE
part of \eqref{nondim:Q1}, i.e., $S(t,\bar S)$ satisfies the ODE system
\begin{equation}\label{S equ}
 \left\{
 \begin{aligned}
 &\frac{d}{dt} S=-  aS+  b\left(S^2-\dfrac{1}{3}\tr(S^2)\Id\right)- c \tr(S^2)S,\\
 &S|_{t=0}=\bar{S}\in \sS_0^{(3)}.
\end{aligned}
\right.
\end{equation}
Consider the convex
set (see for instance \cite{BV04} for a proof of its convexity)
\begin{equation}
\label{convex hull}
\mathcal{C}:=\left\{Q\in \sS_0^{(3)}\Big|\; -\frac{b+\sqrt{b^2-24ac}}{12c}\leq
\lambda_i(Q)\leq \frac{b+\sqrt{b^2-24ac}}{6c},\ \  i=1,2,3\right\}.
\end{equation}
It has been shown in \cite{IXZ14} that the solution $S(t, \bar{S})$
to \eqref{S equ} always lies in this set $\mathcal{C}$ provided that
its initial datum $\bar{S}$ lies in \eqref{convex hull} (for details, see Step 2
of the proof of \cite[Proposition 2.2]{IXZ14}).

On the other hand, for $\bar R\in H^2(\TT;\sS_0^{(3)})$, we denote $R(t; s, \bar R)= \mathcal{V}(t, s) \bar R$ the unique solution to the linear non-autonomous problem
\begin{equation}
\left\{
\begin{aligned}
&\partial_t R- \eps\Delta{R} =-\mathbf{u}\cdot\nabla{R}+\Omega{R}-R\Omega,\\
&R|_{t=s}=\bar R,
\end{aligned}
\right.
\label{R equ}
\end{equation}
where $\mathbf{u}$ is a given divergence-free function in $L^\infty (0,T;W^{1,\infty}(\TT;\R^3))$ and $\Omega=\dfrac{\nabla \mathbf{u}-\nabla^T\mathbf{u}}{2}$.
Next, we prove that
the two parameter evolution system $\mathcal{V}(t, s)$ defined by problem
\eqref{R equ} also preserves the above closed convex hull of the range for
the initial data.

Denote
$$
  m=\frac{b+\sqrt{b^2-24ac}}{12c}.$$
Similar to the singular potential of Ball--Majumdar type considered in \cite{BM10},
for every fixed (sufficiently small) $\delta>0$, we take
\begin{equation}
\label{singular potential}
\psi_\delta(Q)=\displaystyle\inf_{\rho\in\A_{Q,\delta}}\int_{\mathbb{S}^2}\rho(\mathbf{p})\ln{\rho(\mathbf{p})}d\mathbf{p},
\end{equation}
where the admissible set $\A_{Q, \delta}$ is given by
$$
  \A_{Q,\delta}=\left\{\rho:\mathbb{S}^2\rightarrow\mathbb{R},\, \rho\geq 0,\,
  \int_{\mathbb{S}^2}\rho(\mathbf{p})d\mathbf{p}=1;\; Q=3(1+\delta)m\int_{\mathbb{S}^2}\Big(\mathbf{p}\otimes{\mathbf{p}}-\frac13\Id\Big)\rho(\mathbf{p})d\mathbf{p}
  \right\}.
$$
According to \cite{BM10}, for a given $Q$, we minimize the entropy term over all probability distributions $\rho$ that have a fixed normalized second moment $Q$.
Define the bulk potential
$$
 \Psi_\delta(Q)= \begin{cases}
  \psi_\delta(Q), \quad\mbox{if } -(1+\delta)m < \mathrm\lambda_i(Q)< 2(1+\delta)m, \ \ 1\leq
  i\leq 3, \\
  \infty, \qquad\ \ \mbox{otherwise}.
  \end{cases}
$$
Namely, the minimization over the set $\A_{Q,\delta}$ is only defined for those $Q$-tensors that can
be expressed as the normalized second moment of a probability distribution
function $\rho$ and whose eigenvalues obey the above physical constraints.
Similarly to \cite{FRSZ15}, one can easily check that
\begin{itemize}
\item $\psi_\delta(Q)$ is strictly convex.
\item $\psi_\delta(Q)$ is smooth in $\mathrm{Dom}(\Psi_\delta)$, i.e., when
$$
-(1+\delta)m < \mathrm\lambda_i(Q)< 2(1+\delta)m,\ \  1\leq  i\leq 3.
$$
\item $\psi_\delta(Q)$ is isotropic, i.e.,
$$
  \psi_\delta(Q)=\psi_\delta(BQB^T), \quad\forall\, B\in \mathrm{SO}(3).
$$
\end{itemize}
\begin{remark}
Since $\psi_\delta(Q)$ is isotropic and satisfies $\tr(Q)=0$, it follows from
\cite{B84, B12} that
\begin{equation}
\label{isotropy}
\psi_\delta(Q)=h_\delta(|Q|^2, \mathrm{det}Q), \quad\forall\, Q\in
\sS_0^{(3)},
\end{equation}
where $h_\delta$ is smooth in $\mathrm{Dom}(\Psi_\delta)$ (see \cite[Section 5]{B84}). Then we have
\begin{align}\label{isotropy1}
\frac{\partial \psi_\delta(Q)}{\partial Q}
&= \frac{\partial h_\delta}{\partial |Q|^2} \frac{\partial |Q|^2}{\partial Q}+ \frac{\partial h_\delta}{\partial \mathrm{det}Q} \frac{\partial \mathrm{det}Q}{\partial Q}\nonumber\\
&= 2\frac{\partial h_\delta}{\partial |Q|^2}Q+\frac{\partial h_\delta}{\partial \mathrm{det}Q} Q^*,
\end{align}
where $Q^*$ is the adjoint matrix of the symmetric matrix $Q$.
\end{remark}

 Now consider the linear equation  \eqref{R equ} on $[0,T]$.
 Let $T_\delta^\ast \in (0, T]$  be defined as follows
$$
  T_\delta^\ast=\sup\big\{0\leq t< T: \ -(1+\delta)m < \mathrm\lambda_i(R(t, x))<2 (1+\delta)m,\ \forall\, x\in\TT,\ 1\leq
  i\leq 3 \big\}.
$$
We claim that
$$
T_\delta^\ast=T.
$$
This conclusion can be justified by a contradiction argument.
Suppose $T_\delta^\ast<T$, we see that
$\psi_\delta(R(t))\in H^2(\TT)$ for $t\in[0, T_\delta^\ast)$.
Taking the (matrix) inner product of equation \eqref{R equ} with
$\dfrac{\partial\psi_\delta}{\partial{R}}$ yields
$$
  \partial_t\psi_\delta+\mathbf{u}\cdot\nabla\psi_\delta-(\Omega{R}-R\Omega):
  \dfrac{\partial\psi_\delta}{\partial{R}}=\eps\Delta{R}:
  \dfrac{\partial\psi_\delta}{\partial{R}}, \qquad \forall\, (x,t)\in \TT \times [0,  T_\delta^\ast).
$$
Since $(\Omega{R}-R\Omega): R=(\Omega{R}-R\Omega): R^*=0$, then by \eqref{isotropy1} we see
that the above equation can be reduced
to
\begin{align}\label{equ for max principle}
  \partial_t\psi_\delta+\mathbf{u}\cdot\nabla\psi_\delta
  &=\eps\Delta{R}:
  \dfrac{\partial\psi_\delta}{\partial{R}}\non\\
  &=\eps\Delta\psi_\delta-\eps\dfrac{\partial^2\psi_\delta}{\partial{R}_{ij}\partial{R}_{lm}}\partial_k R_{lm}\partial_k R_{ij}\non\\
  &\leq \eps\Delta\psi_\delta,
  \qquad\forall\, (x,t)\in \TT \times [0,  T_\delta^\ast),
\end{align}
where in the last step we used the convexity of
$\psi_\delta(R)$.

Hence, the maximum principle for $\psi_\delta$ together with the assumption
\eqref{initial setting} on the initial eigenvalue-range for $\bar R$ gives
\begin{equation}
  \psi_\delta(R(t, x)) \leq \sup\{\psi_\delta(\bar R(x))\} <
  \infty, \qquad \forall\, (x,t)\in \TT \times [0,  T_\delta^\ast).
  \nonumber
\end{equation}
The strict physicality property of  $\psi_\delta$ (in a similar manner as in \cite[Theorem 1]{BM10}) indicates that
there exists $\varepsilon=\varepsilon(\delta)>0$
\begin{equation}
  -(1+\delta)m+\varepsilon(\delta) < \mathrm\lambda_i(R(x,t))<
  2(1+\delta)m-\varepsilon(\delta), \qquad \forall\, (x,t)\in \TT \times [0,  T_\delta^\ast), \ \ 1\leq i\leq 3.
  \nonumber
\end{equation}
This however leads to a contradiction with the definition of
$T_\delta^\ast$.

As a consequence, we conclude that $T_\delta^\ast=T$ and
\begin{equation}
  -(1+\delta)m < \mathrm\lambda_i(R(x,t))<
  2(1+\delta)m, \qquad \forall\,(x,t)\in \TT \times [0, T), \ \ 1\leq i\leq 3.\nonumber
\end{equation}
Since $\delta>0$ is arbitrary, after performing a continuity argument we can obtain
that
\begin{equation}
  -m \leq \mathrm\lambda_i(R(x,t))\leq 2m, \qquad \forall\,(x,t)\in \TT\times [0, T], \ \ 1\leq i\leq 3.
  \nonumber
\end{equation}

\subsection{Eigenvalue-range preservation under weaker regularity}
\label{subsec:proofofweakflow}
We show that the previous eigenvalue-range
preservation property may still hold for solutions with weaker
regularity (even weaker than what we actually need).

\begin{proposition}\label{lemma:smoothflow}
Let $\mathbf{u}\in L^\infty(0, T; L^2(\TT))\cap L^2(0,T;
H^1(\TT))$, $\mathbf{u}_\delta\in L^\infty (0,T;
W^{1,\infty}(\TT)\cap H^2(\TT))$,
$\nabla\cdot\mathbf{u}=\nabla\cdot\mathbf{u}_\delta=0$ be such that
$$
\mathbf{u}_\delta \to \mathbf{u},\quad \text{strongly in } L^\infty(0,T;L^2(\TT))\cap L^2(0,T;H^1(\TT)) \text{ as }\delta\to 0.
$$
Assume that  $\dQ\in L^\infty(0,T;H^1(\TT))\cap
L^2(0,T;H^2(\TT))$ is the unique solution of the system
\begin{equation}
\left\{
\begin{aligned}
&\dQ_t+\mathbf{u_\delta}\cdot\nabla\dQ-\Omega_\delta\dQ+\dQ\Omega_\delta\non\\
&\quad =\eps \Delta\dQ- a\dQ+b\Big[(\dQ)^2-\frac{1}{3}\tr((\dQ)^2)\Id\Big]- c\dQ\tr^2(\dQ),\\
&\dQ|_{t=0}=Q_0(x),
\end{aligned}
\right.
\label{Q-equdelta}
\end{equation}
and $\dQ$ satisfies
\begin{equation}
\underline{\lambda}_i\leq\lambda_i (\dQ(x,t))\leq\overline{\lambda}_i,\quad (x,t)\in \TT\times[0,T],\ \ i=1,2,3,
\label{eq:eigenvalconstapprox}
\end{equation}
where
$\lambda_i (\dQ)$ denote the increasingly ordered eigenvalues of $\dQ$.

Then $Q^{(\delta)}(x,t)\to Q(x,t)$ a.e. in $\TT\times[0,T]$ as $\delta \to 0$, where $Q$ is the unique solution
in $L^\infty(0,T; H^1(\TT))\cap L^2(0,T; H^2(\TT))$
of the system
\begin{equation}
\left\{
\begin{aligned}
&Q_t+\mathbf{u}\cdot\nabla{Q}-\Omega{Q}+Q\Omega\non\\
&\quad =\veps\Delta{Q}-a Q+ b\Big[Q^2-\frac{1}{3}\tr(Q^2)\Id\Big]- c Q\tr(Q^2),\\
&\dQ|_{t=0}=Q_0(x).
\end{aligned}
\right.
\label{Q-equ+}
\end{equation}
Moreover, we have an eigenvalue constraint on $Q$ (provided that its
initial datum $Q_0$ also has the same constraint):
\begin{equation}
\underline{\lambda}_i\le \lambda_i (Q(x,t))\le \overline{\lambda}_i,\quad \forall\, t\in [0,T],\ \text{a.e. }x\in \TT,\ i=1,2,3.
\label{eq:eigenvalconst}
\end{equation}
\end{proposition}
\begin{proof}
First, we prove  an a priori $L^\infty$-bound for $Q$.
Since $c>0$, there exists a positive number $\eta_0>0$ such that for any $Q\in \sS_0^{(3)}$,
$$
-a\tr(Q^2)+ b\tr(Q^3)- c\tr^2(Q^2)\le 0,\quad \text{if }|Q|\ge \eta_0.
$$
Then we take $\eta:=\max\{\|Q_0\|_{L^\infty},\eta_0\}$. Multiplying
\eqref{Q-equ+} by $Q(|Q|^2-\eta)_+$ and integrating over $\TT$, after integration by parts, we obtain
\begin{align}
\frac{1}{2}\frac{d}{dt}\int_{\TT} (|Q|^2-\eta)_+^2\, dx
&=-\veps\int_{\TT}|\nabla Q|^2(|Q|^2-\eta)_+\,dx-\frac{\veps}{2}\int_{\TT} |\nabla (|Q|^2-\eta)_+|^2\,dx\non\\
&\quad+\int_{\TT} (-a\tr(Q^2)+ b\tr(Q^3)- c\tr^2(Q^2))(|Q|^2-\eta)_+\, dx\non\\
&\le 0,\non
\end{align}
which implies
\begin{equation}\label{Qinftyeta}
\|Q(\cdot, t)\|_{L^\infty(\TT)}\le \eta, \quad \forall\,t\in [0,T].
\end{equation}
In a similar manner, we can obtain the same result for $\dQ$
\begin{equation}\label{dQinftyeta}
\|\dQ(\cdot, t)\|_{L^\infty(\TT)}\le \eta, \quad \forall\,t\in [0,T].
\end{equation}
Denote the difference $\dR:=\dQ-Q$. Then $\dR$  satisfies the following equation
\begin{equation*}
\left\{
\begin{aligned}
&\partial_t \dR+\mathbf{u}_\delta\cdot \nabla \dR-\Omega_\delta \dR+\dR\Omega_\delta\\
&\quad =\eps\Delta\dR- a\dR+(\mathbf{u}-\mathbf{u}_\delta)\cdot \nabla Q-(\Omega-\Omega_\delta)Q+Q(\Omega-\Omega_\delta)\\
&\qquad+ b\left(\dR\dQ+Q\dR-\frac{1}{3}\textrm{tr}\big(\dR\dQ+Q\dR\big)\Id\right)\\
&\qquad- c\left(\dR\textrm{tr}((\dQ)^2)+Q\textrm{tr}\big(\dR\dQ+Q\dR\big)\right),\\
&\dR|_{t=0}=\dR_0=0.
\end{aligned}
\right.
\end{equation*}
Multiply the above equation by $\dR$ and integrating over $\TT$, after integration  by parts  and taking into account the uniform $L^\infty$-bounds on
$Q$, $\dQ$ and also $\dR$, respectively (see \eqref{Qinftyeta}, \eqref{dQinftyeta}), we get
\begin{align}
& \frac{d}{dt}\int_{\TT} |\dR|^2\,dx\non\\
&\quad \le C\int_{\TT}|\dR|^2\,dx+C\int_{\TT}
\big[(\mathbf{u}-\mathbf{u}_\delta)\cdot \nabla Q-(\Omega-\Omega_\delta)Q+Q(\Omega-\Omega_\delta)\big]:\dR\,dx\non\\
&\quad  \le C\|\dR\|_{L^2}^2+C\big(\|\nabla Q\|_{L^2}
\|\mathbf{u}-\mathbf{u}_\delta\|_{L^2}+\|Q\|_{L^2}\|\Omega-\Omega_\delta\|_{L^2}\big)\|\dR\|_{L^\infty},\non\\
&\quad  \le C\|\dR\|_{L^2}^2+C\big(\|\mathbf{u}-\mathbf{u}_\delta\|_{L^2}+\|\Omega-\Omega_\delta\|_{L^2}\big),\non
\end{align}
where $C$ a constant depending only  on $a,b,c,\eta$ and $\|Q\|_{L^\infty(0,T;H^1)}$. Thus, it holds
\begin{align}
\|\dR(t)\|_{L^2}^2 \leq C e^{Ct}\int_0^t e^{-C\tau} \big(\|\mathbf{u}(\tau)-\mathbf{u}_\delta(\tau)\|_{L^2}+\|\Omega(\tau)-\Omega_\delta(\tau)\|_{L^2}\big)\, d\tau,\quad \forall\, t\in [0,T].\non
\end{align}
From our assumptions, we know that
$$
\int_0^t\|\mathbf{u}(\tau)-\mathbf{u}_\delta(\tau)\|_{L^2}^2\,d\tau\to 0, \quad
\int_0^t\|\Omega(\tau)-\Omega_\delta(\tau)\|_{L^2}^2\,d\tau \to 0,\quad \forall\, t\in [0,T]
$$
as $\delta\to 0$. Then we can conclude
\begin{equation}
\lim_{\delta\to 0}\|\dR(t)\|_{L^2}\to 0, \quad \forall\,  t\in [0,T].
\end{equation}
Thus, for any time $t\in [0,T]$, we get $\dQ(x,t)\to Q(x,t)$ a.e. in $\TT$ and thanks to the continuity
 of eigenvalues as functions of matrices (see, for instance \cite{nomizu}) we obtain the claimed bounds on the eigenvalues of $Q$. The proof is complete.
\end{proof}

\subsection{Proof of Theorem \ref{main theorem}}
\label{subsec:proofofmain}
We are in a position to prove Theorem \ref{main theorem}. Returning to equation \eqref{nondim:Q1}, we assume for now $\mathbf{u}\in L^\infty(0,T;W^{1,\infty})\cap L^\infty (0,T;H^2)$, a regularity assumption that will be eventually removed thanks to the previous subsection. For some time $T>0$, $n\in \mathbb{N}$ and $k=\{1,2,...,n\}$, we denote
\begin{align*}
   U_k&\defeq\left(\mathcal{V}\Big(\frac{k}{n}T,\frac{k-1}{n}T\Big)S\Big(\frac{T}{n},\cdot\Big)\right)
  \circ\cdots\left(\mathcal{V}\Big(\frac{2T}{n},\frac{T}{n}\Big)S\Big(\frac{T}{n},\cdot\Big)\right)\circ \left(\mathcal{V}\Big(\frac{T}{n},
  0\Big)S\Big(\frac{T}{n},\cdot\Big)\right)Q_0(x),
\end{align*}
and $U_0=Q_0$. Furthermore, we define
\begin{align}
  U(t)&\defeq \mathcal{V}\Big(\frac{k-1}{n}T+\eta,\frac{k-1}{n}T\Big)S(\eta,\cdot)U_{k-1}, \quad \text{for }t=\frac{k-1}{n}T+\eta\ \ \text{with }\ 0\le \eta<\frac{T}{n}, \label{limit}
\end{align}
for $k=1,...,n$.

Recalling the Trotter product formula, in its abstract form as in Proposition \ref{proposition in T} in the Appendix, we choose
\begin{align*}
&V=H^2(\TT), \quad  W=H^1(\TT),\quad \gamma=\delta=\frac12,
\end{align*}
and
\begin{align*}
&A(t)Q=\eps \Delta Q-\mathbf{u}(\cdot,t)\cdot \nabla Q + \Omega(\cdot,t)Q-Q\Omega(\cdot,t),\\
&X(Q)=-  a{Q}+  b\left(Q^2-\dfrac{1}{3}\tr(Q^2)\Id\right)- c\tr(Q^2)Q.
\end{align*}
Below we proceed to verify the assumptions (H1)--(H2) made in Proposition \ref{proposition in T} under the above choices.

First, we can easily see that the solution operator of the ODE system  \eqref{S equ} fulfills the assumption (H2) on its existence interval $[0,T]$.

Next, for $\mathbf{u}\in L^\infty (0,T; W^{1,\infty} (\TT))$, the existence of the evolution system $\mathcal{V}(t,s)$ associated with the linear parabolic problem \eqref{R equ} with initial datum $Q_0\in V$ can be easily verified. Then it remains to verify the assumption (H1).

Let $Q(t)=\mathcal{V}(t, 0)Q_0(x)$ be solution to \eqref{R equ} with initial datum $Q_0$.
Multiplying both sides of equation \eqref{R equ} with $Q$ and integrating over $\TT$, we have
\begin{align*}
\dfrac12\dfrac{d}{dt}\|Q(t)\|_{L^2}^2
&=\int_{\TT} Q_t: Q\;dx \\
&=-\eps\int_{\TT}|\nabla{Q}|^2\;dx-\frac12\int_{\TT}\mathbf{u}\cdot\nabla{|Q|^2}dx+\int_{\TT}(\Omega{Q}-{Q}\Omega): Q\;dx \\
&=-\eps\int_{\TT}|\nabla{Q}|^2\;dx \leq 0,
\end{align*}
where we used the facts that $\nabla\cdot\mathbf{u}=0$ and
$(\Omega{Q}-Q\Omega): Q=0$. Thus,
\be
\|Q(t)\|_{L^2}\leq\|Q(s)\|_{L^2}, \quad\forall\, 0\leq s< t\leq T.\label{esL2}
\ee
Multiplying both sides of equation \eqref{R equ} with
$\Delta^2{Q}$ and integrating over $\TT$,  we have
\bigskip
\begin{align*}
&\frac12\frac{d}{dt}\|\Delta{Q}(t)\|_{L^2}^2\non\\
&\quad =-\eps\int_{\TT}|\nabla\Delta{Q}|^2\,dx+\int_{\TT}\nabla(\mathbf{u}\cdot\nabla{Q}): \nabla\Delta{Q}\,dx-\int_{\TT} \nabla (\Omega{Q}-{Q}\Omega): \nabla\Delta{Q}\, dx \\
&\quad \leq- \frac{\eps}{2}\int_{\TT}|\nabla\Delta{Q}|^2\,dx+C\int_{\TT}|\nabla\mathbf{u}|^2|\nabla{Q}|^2\,dx
+C\int_{\TT}|\mathbf{u}|^2|\nabla^2Q|^2\,dx+C\int_{\TT}|\nabla^2\mathbf{u}|^2| Q|^2\, dx\\
&\quad \leq-\frac{\eps}{2}\|\nabla\Delta{Q}\|_{L^2}^2+C\|\mathbf{u}\|_{W^{1,\infty}}^2 \Big(\|\nabla^2{Q}\|_{L^2}^2+\|\nabla {Q}\|_{L^2}^2\Big) +C\|\mathbf{u}\|_{H^2}^2\|Q\|_{L^\infty}^2\\
&\quad \leq C\Big(\|\Delta{Q}\|_{L^2}^2+\|{Q}\|_{L^2}^2\Big),
\end{align*}
where the constant $C$ depends on $\|\mathbf{u}\|_{L^\infty(0,T; W^{1,\infty}(\TT))}$ and $\|\mathbf{u}\|_{L^\infty(0,T; H^{2}(\TT))}$.
As a result, we can easily deduce from Gronwall's lemma that
$$
  \|Q(t)\|_{H^2}\leq e^{C(t-s)}\|Q(s)\|_{H^2}, \quad \forall\, 0\leq s<t\leq T,
$$
which implies
\be\label{est:regularAts}
\| \mathcal{V}(t, s)\|_{\LL(V)}\leq e^{C(t-s)}, \quad \forall\, 0\leq s<t\leq T.
\ee

We note that the solution $Q(t)$ can be written as
\begin{align}
Q(t)= e^{(t-s)\eps\Delta} Q(s) +\int_s^{t} e^{(t-\tau)\eps\Delta} ( -\mathbf{u}\cdot \nabla Q +\Omega Q-Q\Omega)(\tau)\, d\tau,\quad  0\leq s<t\leq T.\non
\end{align}
We recall a simple and rather standard estimate for the heat
equation. To this end, we decompose any function $u\in L^2$ as
$u=\sum_{k\in\N}\omega_k u_k $ where $\{\omega_k\}_{k\in\N}$ form a
complete orthonormal system in $L^2$ of eigenvectors of the minus
Laplacian, with corresponding eigenvalues $\{\lambda_k\}_{k\in\N}$
ordered in a non-decreasing sequence. Then we have
$e^{t\Delta}u=\displaystyle\sum_{k\in\N} e^{-\lambda_k t}\omega_k
u_k$ and $-\Delta
\left(e^{t\Delta}u\right)=\displaystyle\sum_{k\in\N} \lambda_k
e^{-\lambda_k t}\omega_k u_k$. It easily follows that
\be
\|\nabla (e^{t\Delta} u)\|_{L^2}^2=(-\Delta (e^{t\Delta}u),
e^{t\Delta}u)=\sum_{k\in\N} e^{-\lambda_k t}\lambda_k (u_k)^2\le
\frac{C}{t}\sum_{k\in\N} (u_k)^2=\frac{C}{t}\|u\|^2, \quad
\text{for}\ t>0. \ee
Using the above estimate and the uniform estimate of $\mathbf{u}$, we obtain
\begin{align}
&\|\Delta Q(t)\|_{L^2}\non\\
&\quad \le C(t-s)^{-\frac12}\|\nabla Q(s)\|_{L^2}
+C\int_s^t (t-\tau)^{-\frac12} \|\nabla(\mathbf{u}(\tau)\cdot \nabla Q(\tau))\|_{L^2}\,d\tau \non\\
&\qquad +C\int_s^t (t-\tau)^{-\frac12} \|\nabla (\Omega(\tau) Q (\tau)-Q (\tau)\Omega(\tau))\|_{L^2}\, d\tau\non\\
&\quad \le C(t-s)^{-\frac12}\|\nabla Q(s)\|_{L^2}
  +C\int_s^t (t-\tau)^{-\frac12} \|\Delta Q(\tau)\|_{L^2}\,d\tau \non\\
&\qquad   +C\int_s^t (t-\tau)^{-\frac12} \|Q(\tau)\|_{H^1} \,d\tau.\label{integral-inequality}
\end{align}
We recall estimate \eqref{esL2}, which together with the simple fact on $\TT$
$$\|\Delta u\|_{L^2}+\|u\|_{L^2}\le C\|u\|_{H^2}\le \tilde C\left(\|\Delta u\|_{L^2}+\|u\|_{L^2}\right)$$
and the last estimate implies
\be
\|Q(t)\|_{H^2}\le C(t-s)^{-\frac12}\| Q(s)\|_{H^1}
  +C\int_s^t (t-\tau)^{-\frac12} \| Q(\tau)\|_{H^2}\,d\tau.
  \label{integral-inequality+}
\ee
Let $$f(t)\defeq (t-s)^{\frac12}\|Q(t)\|_{H^2}.$$
We deduce from the above integral inequality
that
\begin{align}
f(t)\leq  C\|Q(s)\|_{H^1} + C(t-s)^{\frac12}\int_s^t g(\tau) f(\tau) \, d\tau, \quad 0\leq s<t\leq T,
\end{align}
where
\be
g(\tau):=(t-\tau)^{-\frac 12} (\tau-s)^{-\frac 12}.\non
\ee
On the other hand we have:
\begin{align}
\int_s^t g(\tau)\, d\tau& = \int_s^t (\tau-s)^{-\frac12} (t-\tau)^{-\frac12} \, d\tau\non\\
& = \int_s^{\frac{t+s}{2}} (\tau-s)^{-\frac12} (t-\tau)^{-\frac12}+\int_{\frac{t+s}{2}}^t (\tau-s)^{-\frac12} (t-\tau)^{-\frac12} \, d\tau \nonumber\\
& \leq \sqrt{\frac{2}{t-s}} \int_s^{\frac{t+s}{2}} (\tau-s)^{-\frac12}\, d\tau +
\sqrt{\frac{2}{t-s}} \int_{\frac{t+s}{2}}^t (t-\tau)^{-\frac12} \, d\tau \nonumber\\
& = 2\sqrt{\frac{2}{t-s}} \int_0^{\frac{t-s}{2}} \sigma^{-\frac12}\, d\sigma=4.\non
\end{align}
As a consequence, by Gronwall's lemma, we obtain
\begin{align}
f(t)\leq C\|Q(s)\|_{H^1} e^{C(t-s)^{\frac12}\int_s^t g(\tau)\, d\tau}\leq C\|Q(s)\|_{H^1} ,\quad 0\leq s<t\leq T,\non
\end{align}
which implies that
$$
\| Q(t)\|_{H^2}\leq C\|Q(s)\|_{H^1} (t-s)^{-\frac12},\quad 0\leq s<t\leq T.
$$
Combing the above estimate with \eqref{esL2}, we infer that
\be\label{est:singularreg}
\|\mathcal{V}(t, s)\|_{\LL(W, V)}\leq C(T)(t-s)^{-\frac12}, \quad 0\le s<t\le T.
\ee

Next, we have the following estimate
\begin{align}
&\|Q(t)-Q(s)\|_{L^2}\non\\
&\quad \leq  \int_s^t \|\partial_t Q(\tau)\|_{L^2}\, d\tau\nonumber\\
&\quad \leq C \int_s^t \left( \|\Delta{Q}(\tau)\|_{L^2}+\|(\mathbf{u}\cdot\nabla{Q})(\tau)\|_{L^2}+\|(\Omega{Q}-{Q}\Omega)(\tau)\|_{L^2}\right)\, d\tau \non\\
&\quad \leq C \int_s^t \left( \|{Q}(\tau)\|_{H^2}+\|\mathbf{u}(\tau)\|_{L^\infty}\|\nabla{Q}(\tau)\|_{L^2}+\|\nabla \mathbf{u}(\tau)\|_{L^3}\|{Q}(\tau)\|_{L^6}\right)\, d\tau\non\\
&\quad \leq C \int_s^t \|{Q}(\tau)\|_{H^2} \, d\tau\non\\
&\quad \leq Ce^{CT}(t-s) \|{Q}(s)\|_{H^2},\quad 0\leq s<t\leq T,\non
\end{align}
where $C$ may depend on $\|\mathbf{u}\|_{L^\infty(0,T;W^{1,\infty})}$.
The last estimate implies
$$
 \|\mathcal{V}(t, s)-I\|_{\LL(H^2, L^2)}\leq C(T)(t-s), \quad\ 0\leq s<t\leq T.
$$
On the other hand, from \eqref{est:regularAts} we already know
$$
 \|\mathcal{V}(t, s)-I\|_{\LL(H^2, H^2)}\leq C(T),\quad\ 0\leq s<t\leq T.
$$
Hence, from the interpolation inequality  $\|f\|_{H^1}\le C \|f\|_{L^2}^{\frac 12}\|f\|_{H^2}^{\frac 12}$ for any $f\in H^2(\TT)$, we easily conclude that
\be\label{eq:dualityinterpo}
 \|\mathcal{V}(t, s)-I\|_{\LL(V, W)}\leq C(t-s)^{\frac12},\quad\ 0\leq s<t\leq T.
\ee

In view of the estimates \eqref{est:regularAts}, \eqref{est:singularreg} and \eqref{eq:dualityinterpo}, the evolution system $\mathcal{V}(t, s)$ satisfies \eqref{property of A}. Thus, the assumption (H1) is verified.

In summary, we are able to apply Proposition \ref{proposition in T} to conclude that
\begin{equation}
 \|Q(t)-U(t)\|_{H^2}\le Cn^{-\frac12}, \quad \;\forall\; t\in [0,T].\label{conv}
\end{equation}
On the other hand, we observe that $U_k$ can be viewed as the successive superpositions
of solutions of the ODE part \eqref{S equ} and the second order
parabolic equation part \eqref{R equ} with specific initial data.
As a consequence, since we have already known that both operators $S(t,\cdot)$ and
$\mathcal{V}(t, s)$ preserve the closed convex hull of the range of
the initial data, then the proof of Theorem \ref{main theorem} can
be accomplished by passing to the limit $n\to +\infty$ in
\eqref{limit}, which is guaranteed by the estimate \eqref{conv}
(where thanks to the previous subsection we do use the result on
eigenvalue preservation for $\mathcal{V}(t,s)$ at the level of
regularity on $\mathbf{u}$ that is guaranteed by the strong global
solutions).

The proof of Theorem \ref{main theorem} is complete. $\Box$

\begin{remark}
Theorem~\ref{main theorem} is still valid if the two dimensional torus
setting is replaced by the whole space.
\end{remark}

\section{Large Ericksen number limit in the co-rotational case}
\label{sec:highEricksen}
\setcounter{equation}{0}

In this section we prove Theorem~\ref{thm:highEricksen}.
To this end, we denote by $(\mathbf{v},R)$ the solution of the limit system obtained by formally setting $\veps=0$ in the system \eqref{nondim:u1}--\eqref{nondim:Q1} with $\xi=0$. Hence, it satisfies the following {\it weakly coupled } system:
\begin{align}
&\delt \mathbf{v}+ \mathbf{v}\cdot\nabla\mathbf{v}+\nabla q= 0,\label{nondim:v10-}\\
&\nabla\cdot \mathbf{v}= 0,\label{nondim:dv-}\\
&\delt R+\mathbf{v}\cdot\nabla R=\Omega_{\mathbf{v}}R-R\Omega_{\mathbf{v}}-aR+b\big(R^2-\frac{1}{3}\tr(R^2)\Id\big)- cR\tr(R^2),
\label{nondimn:R10-}\\
&\mathbf{v}|_{t=0}=\mathbf{u}_0,\quad R|_{t=0}=Q_0,\label{nondim:ini-}
\end{align}
where $\Omega_\mathbf{v}:=\displaystyle{\frac{\nabla \mathbf{v}-\nabla^T \mathbf{v}}{2}}$.

 Since it is assumed that $\mathbf{u}_0\in H^{5}$, thanks to the regularity theory for $2D$ Euler equation we have $\mathbf{v}\in L^\infty (0,T; H^{5})$ (see e.g. \cite{kato}).
 Next, by standard results for transport equation on the torus with $\nabla \cdot \mathbf{v}=0$, due to the assumption on the initial datum $Q_0\in H^{4}$ and
 the aforementioned regularity of $\mathbf{v}$, we are able to deduce  $R\in L^\infty(0,T;H^{4})$ (see Proposition~\ref{prop:globalexistlimitsyst}  in the Appendix for further details).

For any $\eps>0$, we denote $(\mathbf{u}^\eps, Q^\eps)$ the global strong solution to system \eqref{nondim:u1}--\eqref{nondim:Q1} with $\xi=0$, subject to the initial data $\mathbf{u}^\eps|_{t=0}=\mathbf{u}_0$ and $Q^\eps|_{t=0}=Q_0$.

Set
 $$\mathbf{w}^\eps=\mathbf{u}^\eps-\mathbf{v},\quad  S^\eps=Q^\eps-R.$$
Then the difference functions $(\mathbf{w}^\eps,S^\eps)$ satisfy the following system (dropping  the superscript $\eps$ for the sake of simplicity):

\begin{equation}
\left\{
   \begin{aligned}
    &\partial_t \mathbf{w}+\mathbf{w}\cdot\nabla \mathbf{w} -\veps\Delta \mathbf{w}+ \nabla(P-q)\\
       &\hspace{5ex} = -(\mathbf{v}\cdot\nabla \mathbf{w}+\mathbf{w}\cdot\nabla \mathbf{v})
         -\veps^2\nabla\cdot(\nabla S\odot\nabla S)\\
       &\hspace{8ex} +\veps^2\nabla\cdot\big(S\Delta S-(\Delta S) S\big)
       +\veps^2 \nabla\cdot\big(S\Delta R+R\Delta S-(\Delta S) R-(\Delta R)S\big)\\
       &\hspace{8ex} -\veps^2 \nabla\cdot\big(\nabla S\odot\nabla R+\nabla R\odot\nabla S\big) +\veps\mathcal{F}_\veps(\mathbf{v},R),\\
       &\nabla \cdot \mathbf{w}=0,\\
       & \partial_t S+ \mathbf{w}\cdot \nabla S-\Omega_\mathbf{w} S +S\Omega_\mathbf{w}-\veps\Delta S\\
       &\hspace{5ex}= -(\mathbf{w}\cdot \nabla R+ \mathbf{v}\cdot \nabla S)+\Omega_\mathbf{w} R+\Omega_\mathbf{v} S -R\Omega_\mathbf{w} -S\Omega_\mathbf{v}\\
       &\hspace{8ex} -aS +b\big[SQ+RS-\frac{1}{3}\textrm{tr}\big(SQ+RS\big)\Id \big]\\
       &\hspace{8ex} - c\left[S\textrm{tr}(Q^2)+R\textrm{tr}(QS+SR)\right]+\eps \Delta R,\\
       & \mathbf{w}|_{t=0}=\mathbf{0},\quad S|_{t=0}=0,
     \end{aligned}\right.
      \label{approxsystemdif}
\end{equation}
  where  $\Omega_\mathbf{w}:=\displaystyle{\frac{\nabla^T \mathbf{w}-\nabla \mathbf{w}}{2}}$ and
\be
  \mathcal{F}_\veps(\mathbf{v},R):=\Delta \mathbf{v}-\veps\nabla\cdot\big(\nabla R\odot \nabla R+\Delta R  R-R\Delta R\big).\non
\ee

 We multiply the third equation in \eqref{approxsystemdif} for $S$ by $-\veps^2\Delta S+\veps S$, integrate over $\T^2$, take the trace and sum with the first equation in \eqref{approxsystemdif} multiplied by $\mathbf{w}$ and integrated over $\T^2$.
 After integration by parts and taking into account the cancellations at the energy level, we obtain
 \begin{align}
 &\frac{d}{dt}\int_{\T^2}\left(\frac{1}{2}|\mathbf{w}|^2+\frac{\veps^2}{2}|\nabla S|^2+\frac{\veps}{2}|S|^2\right)\,dx
 +\int_{\T^2} \Big(\veps |\nabla \mathbf{w}|^2+\veps^3 |\Delta S|^2+ \veps^2|\nabla S|^2\Big) \,dx \non\\
 &\quad = \underbrace{\veps^2\int_{\T^2}\textrm{tr}\left([\mathbf{w}\cdot\nabla S]\Delta S\right)\,dx}_{A_1}
          +\underbrace{\veps^2\int_{\T^2}\textrm{tr}\left([S\Omega_\mathbf{w}-\Omega_\mathbf{w} S]\Delta S\right)}_{A_2}\,dx\non\\
 &\qquad  +\veps^2\int_{\T^2}\textrm{tr}\left([\mathbf{w}\cdot\nabla R]\Delta S\right)\, dx
          + \veps^2\int_{\T^2} \textrm{tr}\left([\mathbf{v}\cdot\nabla S]\Delta S\right)\,dx \non\\
 &\qquad +\veps^2\int_{\T^2}\textrm{tr}\left([S\Omega_\mathbf{v}-\Omega_\mathbf{v}S]\Delta S\right)\,dx
         +\underbrace{\veps^2\int_{\T^2}\textrm{tr}\left([R\Omega_\mathbf{w}-\Omega_\mathbf{w} R]\Delta S\right)\,dx}_{A_3}\non\\
 &\qquad -\underbrace{\veps\int_{\T^2} \textrm{tr}([\mathbf{w}\cdot\nabla S]S)\,dx}_{I_1}
         -\underbrace{\veps\int_{\T^2} \textrm{tr}([S\Omega_\mathbf{w}-\Omega_\mathbf{w} S]S)\,dx}_{I_2} \non\\
 &\qquad -\veps\int_{\T^2}\textrm{tr}([\mathbf{w}\cdot\nabla R] S)\,dx
         -\underbrace{\veps\int_{\T^2} \textrm{tr}([\mathbf{v}\cdot\nabla S]S)\,dx}_{I_3}\non\\
 &\qquad +\veps\int_{\T^2}\textrm{tr}([\Omega_\mathbf{w} R-R\Omega_\mathbf{w}] S)\,dx
         +\underbrace{\veps\int_{\T^2}\textrm{tr}([\Omega_\mathbf{v} S-S\Omega_\mathbf{v}]S)\,dx}_{I_4} \\
 &\qquad \underbrace{-\int_{\T^2}(\mathbf{w}\cdot\nabla \mathbf{w})\cdot \mathbf{w}\, dx-\int_{\T^2}\nabla(P-q)\cdot \mathbf{w}\, dx}_{I_5}\non\\
 &\qquad \underbrace{-\int_{\T^2}(\mathbf{v}\cdot\nabla \mathbf{w})\cdot \mathbf{w}\,dx}_{I_6}
         -\int_{\T^2}(\mathbf{w}\cdot\nabla \mathbf{v})\cdot \mathbf{w}\,dx
         -\underbrace{\veps^2\int_{\T^2}[\nabla\cdot(\nabla S\odot\nabla S)]\cdot \mathbf{w}\, dx}_{A_1'}\non\\
 &\qquad -\underbrace{\veps^2\int_{\T^2}\textrm{tr}([S\Delta S-\Delta S S]\nabla \mathbf{w})\, dx}_{A_2'}
         -\veps^2\int_{\T^2} \textrm{tr}([S\Delta R-\Delta R S]\nabla \mathbf{w})\,dx\non\\
 &\qquad -\underbrace{\veps^2\int_{\T^2} \textrm{tr}([R\Delta S-\Delta S R]\nabla \mathbf{w})\,dx}_{A_3'}
         +\veps^2\int_{\T^2}\textrm{tr}([\nabla S\odot\nabla R+\nabla R\odot\nabla S]\nabla \mathbf{w})\,dx\non\\
 &\qquad -\veps a\int_{\T^2}|S|^2\,dx
         + \veps b\int_{\T^2}\textrm{tr}\big(S QS+R S^2\big)\,dx
        -\veps c\int_{\T^2}\textrm{tr}(Q^2)|S|^2\,dx\non\\
 &\qquad -\veps c\int_{\T^2}\textrm{tr}(R S)\textrm{tr}(QS+S R)\,dx
         + \veps^2a\int_{\T^2}|\nabla S|^2\,dx\non\\
 &\qquad -\veps^2b\int_{\T^2}\textrm{tr}\left([S Q +R S]\Delta S\right)\,dx
         +\veps^2c\int_{\T^2}\textrm{tr}\big(S\Delta S\big)\textrm{tr}(Q^2)\,dx\non\\
 &\qquad +\veps^2c\int_{\T^2}\textrm{tr}(R\Delta S)\textrm{tr}(QS+S R)\,dx\non\\
 &\qquad +\veps\int_{\T^2}\mathcal{F}_\veps(\mathbf{v},R)\cdot \mathbf{w}\, dx+\veps\int_{\T^2}\textrm{tr}(\Delta R(-\eps^2\Delta S+\eps S))\,dx,
 \label{diffes1}
 \end{align}
Using \cite[Lemma 2.1]{ADL15} and the incompressibility condition, one can easily check that
$$A_1=A_1',\quad A_2=A_2',\quad A_3=A_3',$$
as well as
$$I_1=I_2=I_3=I_4=I_5=I_6=0.$$
Then using the H\"older inequality and Young' inequality, we can estimate the remaining terms on the right-hand side of \eqref{diffes1} as follows:
\begin{align}
&\frac12 \frac{d}{dt}\int_{\T^2}\Big(|\mathbf{w}|^2+\veps^2|\nabla S|^2 +\veps|S|^2\Big)\,dx
 +\int_{\T^2}\Big(\veps |\nabla \mathbf{w}|^2 +\veps^3|\Delta S|^2+ \veps^2|\nabla S|^2\Big)\,dx\non\\
&\quad \le \frac{\veps^3}{8}\|\Delta S\|_{L^2}^2
           +C\veps \|\nabla R\|_{L^\infty}^2\|\mathbf{w}\|_{L^2}^2
           +\veps^2 \|\nabla \mathbf{v}\|_{L^\infty}  \|\nabla S\|_{L^2}^2
           +C\veps \|\nabla \mathbf{v}\|_{L^\infty}^2\|S\|_{L^2}^2\non\\
&\qquad +\|\mathbf{w}\|_{L^2}^2 +\veps^2\|\nabla R\|_{L^\infty}^2\|S\|_{L^2}^2
        +\frac{\veps}{8}\|\nabla \mathbf{w}\|_{L^2}^2
        +C\veps \|R\|_{L^\infty}^2\|S\|_{L^2}^2
        +\|\nabla \mathbf{v}\|_{L^\infty}\|\mathbf{w}\|_{L^2}^2\non\\
&\qquad +\frac{\veps}{8}\|\nabla \mathbf{w}\|_{L^2}^2
        +\veps^3 \left( \|\nabla R\|_{L^\infty}^2+\|\Delta R\|_{L^\infty}^2\right)(\left \|\nabla S\|_{L^2}^2+\|S\|_{L^2}^2\right)\non\\
&\qquad +C\veps(\|Q\|_{L^\infty}^2+\|R\|_{L^\infty}^2+1)\|S\|_{L^2}^2+|a|\veps^2\|\nabla S\|_{L^2}^2\non\\
&\qquad +\frac{\veps^3}{8}\|\Delta S\|_{L^2}^2+C\veps(\|Q\|_{L^\infty}^4+\|R\|_{L^\infty}^4)\|S\|_{L^2}^2\non\\\
&\qquad
+\veps^2\|\mathcal{F}_\veps(\mathbf{v},R)\|_{L^2}^2+\|\mathbf{w}\|_{L^2}^2+\frac{\veps^3}{8}\|\Delta S\|_{L^2}^2+\eps\|S\|_{L^2}^2+C\eps^3\|\Delta R\|_{L^2}^2,
\label{diff}
\end{align}
where $C$ is a positive constant independent of $\eps$.

Denote
\begin{align}
\mathcal{Y}(t)&=\|\mathbf{w}\|_{L^2}^2+\veps^2\|\nabla S\|_{L^2}^2 +\veps\|S\|_{L^2}^2,\non\\
\mathcal{H}(t)&= 1+ \|\nabla \mathbf{v}\|_{L^\infty}^2 + (\veps+\veps^2) (\|\nabla R\|_{L^\infty}^2+  \|\Delta R\|_{L^\infty}^2) +\|Q\|_{L^\infty}^4+\|R\|_{L^\infty}^4.\non
\end{align}
It follows from \eqref{diff} and the Cauchy--Schwarz inequality that
\begin{align}
&\frac{d}{dt}\mathcal{Y}(t)+\veps \|\nabla \mathbf{w}\|_{L^2}^2 +\veps^3\|\Delta S\|_{L^2}^2+ \veps^2\|\nabla S\|_{L^2}^2\non\\
&\quad \leq C\mathcal{H}(t)\mathcal{Y}(t)+C\veps^2\|\mathcal{F}_\veps(\mathbf{v},R)\|_{L^2}^2+C\eps^3\|\Delta R\|_{L^2}^2,\non
\end{align}
where $C$ is a constant independent of $\veps$.

From our assumption, we infer that $\mathcal{H}(t)$, $\|\mathcal{F}_\veps(\mathbf{v},R)\|_{L^2}$ and $\|\Delta R\|_{L^2}$ are uniformly  bounded on $[0,T]$ for $0\leq \veps\leq 1$. Since $\mathcal{Y}(0)=0$, it follows from Gronwall's lemma that for $0<\veps< 1$
\begin{align}
&\mathcal{Y}(t) \leq C_T \veps^2, \quad \forall\, t\in [0,T],\non\\
& \int_0^T\left(\veps \|\nabla \mathbf{w}\|_{L^2}^2 +\veps^3\|\Delta S\|_{L^2}^2+ \veps^2\|\nabla S\|_{L^2}^2\right)dt\leq C_T \veps^2.\non
\end{align}

The proof of Theorem~\ref{thm:highEricksen} is complete. \hfill $\Box$

\begin{remark}
The above proof essentially relies on the uniform $L^\infty$-bound of $Q^\eps$, which is the solution of the system \eqref{nondim:u1}--\eqref{nondim:Q1}.
This property is true provided that $\xi=0$ (cf. the proof of Proposition \ref{lemma:smoothflow}). However, for the general case $\xi\neq 0$, it is not clear whether such uniform estimate still holds.
\end{remark}

\section{Lack of eigenvalue-range preservation in the non-corotational case}
\label{sec:lacknoncorot}
\setcounter{equation}{0}

\subsection{The high Ericksen number limit in the non-corrotational case}
As a first step for the proof of Theorem \ref{thm:lackofeigenv}, we prove an analogue of Theorem~\ref{thm:highEricksen} by \emph{assuming} that a uniform a priori $L^\infty$-bound for $Q^\eps$ is available (see \eqref{ass:linftybdeps}).
This assumption will be satisfied if we impose certain bound on the eigenvalues of $Q$ (for instance, \eqref{initial setting+a} below) in the subsequent contradiction argument.

\begin{proposition}\label{thm:highEricksen2} {\bf [High Ericksen number limit in the non co-rotational case].}

Let $\textbf{u}_0\in H^{5}(\TT; \mathbb{R}^3)$  with $\dv \mathbf{u}_0=0$ and $Q_0\in H^{4}(\TT;\sS_0^{(3)})$.
Consider the {\it ``limit  system"}:
\begin{align}
&\delt \mathbf{v}+ \mathbf{v}\cdot\nabla\mathbf{v}+\nabla q=0,\label{nondim:v10x}\\
&\nabla\cdot \mathbf{v}=0,\label{nondim:dvx}\\
&\delt R+\mathbf{v}\cdot\nabla R=(\xi D_{\mathbf{v}}+ \Omega_{\mathbf{v}})\big( R+\frac13 \Id
\big)+\big(R+\frac13 \Id \big)(\xi
D_{\mathbf{v}}- \Omega_{\mathbf{v}}) -2\xi\big(R+\frac13 \Id \big)\tr(R\nabla \mathbf{v})\non\\
&\qquad \qquad \qquad \qquad -aR+b\big[R^2-\frac{1}{3}\tr(R^2)\Id\big]- cR\tr(R^2),
\label{nondimn:R10x}\\
&\mathbf{v}|_{t=0}=\mathbf{u}_0,\quad R|_{t=0}=Q_0,\label{nondim:inix}
\end{align}
where $\displaystyle{\Omega_\mathbf{v}:=\frac{\nabla \mathbf{v}-\nabla^T \mathbf{v}}{2}}$, $\displaystyle{D_\mathbf{v}:=\frac{\nabla \mathbf{v}+\nabla^T \mathbf{v}}{2}}$.
On the other hand, let $(\mathbf{u}^\eps,Q^\eps)$ be the strong solution to system \eqref{nondim:u1}--\eqref{nondim:Q1} starting from the initial data $(\mathbf{u}_0, Q_0)$ on $[0,T]$.
We assume in addition that
\be
\label{ass:linftybdeps}
\|Q^\eps\|_{L^\infty(0,T;\TT)}\le \eta,
\ee
for some positive constant $\eta$ independent of $\eps>0$.  Denote
$$\mathbf{w}^\eps=\mathbf{u}^\eps-\mathbf{v},\quad  S^\eps=Q^\eps-R.$$
Then for all $\eps\in (0,1)$ we have
\begin{align}
&\|\mathbf{w}^\eps(t)\|_{L^2}^2+\veps^2\|\nabla S^\eps(t)\|_{L^2}^2 +\veps\|S^\eps(t)\|_{L^2}^2, \leq C_T \veps^2, \quad \forall\, t\in [0,T],\non\\
& \int_0^T\left(\veps \|\nabla \mathbf{w}^\eps(t)\|_{L^2}^2 +\veps^3\|\Delta S^\eps(t)\|_{L^2}^2+ \veps^2\|\nabla S^\eps(t)\|_{L^2}^2\right)\,dt\leq C_T \veps^2.\non
\end{align}
\end{proposition}
\begin{proof}
The proof is basically an extension of the one for the co-rotational case in Section 4, which however is more involved since $\xi\neq 0$. The reader can safely skip to continue with the proof of Theorem~\ref{thm:lackofeigenv} in the next subsection.

 Like before, since $\mathbf{u}_0\in H^{5}$, thanks to the regularity theory for $2D$ Euler equation we have $\mathbf{v}\in L^\infty (0,T; H^{5})$ (cf. \cite{kato}).
 By standard results for transport equation on the torus with $\nabla \cdot \mathbf{v}=0$, due to the assumption on the initial datum $R_0=Q_0\in H^{4}$ and
 the aforementioned regularity of $\mathbf{v}$, we are able to deduce $R\in L^\infty(0,T;H^{4})$ (See Proposition~\ref{prop:globalexistlimitsyst}). We remark that here the fact $\xi\neq 0$ does not introduce any extra difficulty.

For any $\eps>0$, we denote $(\mathbf{u}^\eps, Q^\eps)$ the global strong solution to system \eqref{nondim:u1}--\eqref{nondim:Q1} with $\xi\in \mathbb{R}$, subject to the initial data $\mathbf{u}^\eps|_{t=0}=\mathbf{u}_0$ and $Q^\eps|_{t=0}=Q_0$.

Set
 $$\mathbf{w}^\eps=\mathbf{u}^\eps-\mathbf{v},\quad  S^\eps=Q^\eps-R.$$
We see that $(\mathbf{w}^\eps,S^\eps)$ satisfies the following system (dropping  the superscript $\eps$ for the sake of simplicity):
\begin{equation}
\left\{
   \begin{aligned}
    &\partial_t \mathbf{w}+\mathbf{w}\cdot\nabla \mathbf{w} -\veps\Delta \mathbf{w}+ \nabla(P-q)\\
       &\hspace{3ex} = -(\mathbf{v}\cdot\nabla \mathbf{w}+\mathbf{w}\cdot\nabla \mathbf{v})
         -\veps^2\nabla\cdot(\nabla S\odot\nabla S)
         +\veps^2\nabla\cdot\big(S\Delta S-(\Delta S) S\big)\\
       &\hspace{6ex} +\veps^2\nabla\cdot\big(S\Delta R+R\Delta S-(\Delta S) R-(\Delta R)S\big)
        -\veps^2\nabla\cdot\big(\nabla S\odot\nabla R+\nabla R\odot\nabla S\big)\\
       &\hspace{6ex} -\veps^2\xi\nabla\cdot\big((\Delta{S})S+S\Delta{S}+\frac23\Delta{S}\big)
         -\veps^2\xi\nabla\cdot\big(\Delta{S}R+R\Delta{S}\big)-\veps^2\xi\nabla\cdot\big(S\Delta{R}+\Delta{R}S\big) \\
       &\hspace{6ex}
       +\frac{2}{3}\veps^2\xi\nabla\cdot\big[( \textrm{tr}(S\Delta S)+\textrm{tr}(R\Delta S)+ \textrm{tr}(S\Delta R))\Id\big]\\
       &\hspace{6ex} +2\veps^2\xi\nabla\cdot\Big[S\,\textrm{tr}(S\Delta{S})+R\,\textrm{tr}(R\Delta{S})+S\,\textrm{tr}(R\Delta{S})
         +R\,\textrm{tr}(S\Delta{S})\Big]\\
       &\hspace{6ex}   +2\veps^2\xi\nabla\cdot\Big[S\,\textrm{tr}(S\Delta{R})+R\,\textrm{tr}(S\Delta{R})+S\,\textrm{tr}(R\Delta{R})\Big]\\
       &\hspace{6ex} +\veps\nabla\cdot\mathcal{G}(Q)+\veps\mathcal{F}_\veps(\mathbf{v},R),\\
       &\nabla \cdot\mathbf{w}=0,\\
       &\partial_t S+ \mathbf{w}\cdot \nabla S-\Omega_\mathbf{w} S +S\Omega_\mathbf{w}-\veps\Delta S\\
       &\hspace{3ex}= -(\mathbf{w}\cdot \nabla R+ \mathbf{v}\cdot \nabla S)+\Omega_\mathbf{w} R+\Omega_\mathbf{v} S -R\Omega_\mathbf{w} -S\Omega_\mathbf{v}\\
       &\hspace{6ex} -aS +b\big(SQ+RS-\frac{1}{3}\textrm{tr}(SQ+RS)\Id\big)
       -c\big[S\textrm{tr}(Q^2)+R\,\textrm{tr}(QS+SR)\big]\\
       &\hspace{6ex} +\xi\big(D_{\mathbf{w}}S+SD_{\mathbf{w}}+\frac23D_{\mathbf{w}}\big)+\xi(D_{\mathbf{w}}R+RD_{\mathbf{w}})+\xi(D_{\mathbf{v}}S+SD_{\mathbf{v}}) \\
       &\hspace{6ex} -\frac23\xi\big[\textrm{tr}(S\nabla\mathbf{w})+\textrm{tr}(S\nabla\mathbf{v})+\textrm{tr}(R\nabla\mathbf{w})\big]\Id
                     -2\xi\big[S\,\textrm{tr}(S\nabla{\mathbf{v}})+S\,\textrm{tr}(R\nabla{\mathbf{v}})+R\,\textrm{tr}(S\nabla{\mathbf{v}}) \big] \\
       &\hspace{6ex} -2\xi\big[S\,\textrm{tr}(S\nabla{\mathbf{w}})+S\,\textrm{tr}(R\nabla{\mathbf{w}})
                      +R\,\textrm{tr}(S\nabla{\mathbf{w}})+R\,\textrm{tr}(R\nabla{\mathbf{w}})\big]+\eps\Delta R,\\
       &\mathbf{w}|_{t=0}=\mathbf{0},\quad S|_{t=0}=0.
       \end{aligned}\right.
      \label{approxsystemdifa}
\end{equation}
  where $\Omega_\mathbf{w}:=\displaystyle{\frac{\nabla \mathbf{w}-\nabla^T \mathbf{w}}{2}}$, $D_\mathbf{w}:=\displaystyle{\frac{\nabla \mathbf{w}+\nabla^T \mathbf{w}}{2}}$  and
\begin{align*}
\mathcal{G}(Q)
  &:=-2\kappa\xi\big(Q+\frac{1}{3}\Id\big)\big(a\textrm{tr}(Q^2)-b\tr(Q^3)+c\textrm{tr}^2(Q^2)\big)\non\\
  &\qquad +2\kappa\xi\big(Q+\frac{1}{3}\Id\big)\big(aQ-b(Q^2-\frac{1}{3}\tr(Q^2)\Id)+cQ\tr(Q^2)\big),\\
\mathcal{F}_\veps(\mathbf{v},R)
  &:=\Delta \mathbf{v}-\veps\nabla\cdot\big(\nabla R\odot \nabla R+\Delta R  R-R\Delta{R}\big)\non\\
  &\qquad -\veps\xi\nabla\cdot\Big((\Delta{R})R+R\Delta{R}+\frac{2}{3}\Delta R-2(R+\frac{1}{3}\Id)(R:\Delta{R}) \Big).
\end{align*}
 We multiply the third equation in \eqref{approxsystemdifa} for $S$ by $-\veps^2\Delta S+\veps S$, integrate over $\TT$, take the trace and sum with the first equation in \eqref{approxsystemdifa} multiplied by $\mathbf{w}$
 and integrated over $\TT$.
 After integration by parts and taking into account the cancellations at the energy level, we obtain
 \begin{align*}
 &\frac{d}{dt}\int_{\TT}\left(\frac{1}{2}|\mathbf{w}|^2
 +\frac{\veps^2}{2}|\nabla S|^2+\frac{\veps}{2}|S|^2\right)\,dx
 +\int_{\TT} \Big(\veps |\nabla \mathbf{w}|^2+\veps^3 |\Delta S|^2
 + \veps^2|\nabla S|^2\Big) \,dx\\
 &= \underbrace{\veps^2\int_{\TT}\textrm{tr}\left([\mathbf{w}\cdot\nabla S]\Delta S\right)\,dx}_{A_1}
          +\underbrace{\veps^2\int_{\TT}\textrm{tr}\left([S\Omega_\mathbf{w}-\Omega_\mathbf{w} S]\Delta S\right)}_{A_2}\,dx
          +\veps^2\int_{\TT}\textrm{tr}\left([\mathbf{w}\cdot\nabla R]\Delta{S}\right)\,dx \\
 &\quad   + \veps^2\int_{\TT} \textrm{tr}\left([\mathbf{v}\cdot\nabla S]\Delta S\right)\,dx
          + \veps^2\int_{\TT}\textrm{tr}\left([S\Omega_\mathbf{v}-\Omega_\mathbf{v}S]\Delta{S}\right)\,dx
          \\
 &\quad +\underbrace{\veps^2\int_{\TT}\textrm{tr}\left([R\Omega_\mathbf{w}-\Omega_\mathbf{w} R]\Delta S\right)\,dx}_{A_3} -\underbrace{\veps^2\xi\int_{\TT}\textrm{tr}\Big(\big(D_{\mathbf{w}}S+SD_{\mathbf{w}}+\frac23D_{\mathbf{w}}\big)\Delta{S}\Big)\,dx}_{A_4}
          \\
 &\quad   -\underbrace{\veps^2\xi\int_{\TT}\textrm{tr}\big([D_{\mathbf{w}}R+RD_{\mathbf{w}}]\Delta{S}\big)\,dx}_{A_5}
 -\veps^2\xi\int_{\TT}\textrm{tr}\big([D_{\mathbf{v}}S+SD_{\mathbf{v}}]\Delta{S}\big)\,dx\\
 &\quad
          +\underbrace{\frac23\veps^2\xi\int_{\TT}\big[\textrm{tr}(S\nabla\mathbf{w})+\textrm{tr}(S\nabla\mathbf{v})+\textrm{tr}(R\nabla\mathbf{w})\big]\mathrm{tr}(\Delta{S})\,dx}_{I_1}\\
 &\quad   +2\veps^2\xi\int_{\TT}\textrm{tr}\Big(\big[S\,\textrm{tr}(S\nabla{\mathbf{v}})+S\,\textrm{tr}(R\nabla{\mathbf{v}})+R\,\textrm{tr}(S\nabla{\mathbf{v}})\big]\Delta{S}\Big)\,dx\\
 &\quad  +\underbrace{2\veps^2\xi\int_{\TT}\textrm{tr}\Big(\big[S\,\textrm{tr}(S\nabla{\mathbf{w}})
             +S\,\textrm{tr}(R\nabla{\mathbf{w}})+R\,\textrm{tr}(S\nabla{\mathbf{w}})+R\,\textrm{tr}(R\nabla{\mathbf{w}})\big]\Delta{S}\Big)\,dx}_{A_6}\\
 &\quad  -\underbrace{\veps\int_{\TT} \textrm{tr}([\mathbf{w}\cdot\nabla S]S)\,dx
         -\veps\int_{\TT} \textrm{tr}([S\Omega_\mathbf{w}-\Omega_\mathbf{w} S]S)\,dx}_{I_2}
         -\veps\int_{\TT}\textrm{tr}([\mathbf{w}\cdot\nabla R] S)\,dx \\
 &\quad -\underbrace{\veps\int_{\TT} \textrm{tr}([\mathbf{v}\cdot\nabla S]S)\,dx}_{I_3}
         +\veps\int_{\TT}\textrm{tr}([\Omega_\mathbf{w}R-R\Omega_\mathbf{w}] S)\,dx
         +\underbrace{\veps\int_{\TT}\textrm{tr}([\Omega_\mathbf{v} S-S\Omega_\mathbf{v}]S)\,dx}_{I_4} \\
 &\quad  -\underbrace{\int_{\TT}(\mathbf{w}\cdot\nabla \mathbf{w})\cdot \mathbf{w}\, dx
         -\int_{\TT}\nabla(P-q)\cdot \mathbf{w}\, dx}_{I_5}
         -\underbrace{\int_{\TT}(\mathbf{v}\cdot\nabla \mathbf{w})\cdot \mathbf{w}\,dx}_{I_6}
         -\int_{\TT}(\mathbf{w}\cdot\nabla \mathbf{v})\cdot \mathbf{w}\,dx\\
 &\quad  -\underbrace{\veps^2\int_{\TT}[\nabla\cdot(\nabla S\odot\nabla S)]\cdot \mathbf{w}\, dx}_{A_1'}
         -\underbrace{\veps^2\int_{\TT}\textrm{tr}([S\Delta S-(\Delta S) S]\nabla \mathbf{w})\, dx}_{A_2'}\\
 &\quad  -\veps^2\int_{\TT} \textrm{tr}([S\Delta R-(\Delta R) S]\nabla \mathbf{w})\,dx
         -\underbrace{\veps^2\int_{\TT} \textrm{tr}([R\Delta S-(\Delta S) R]\nabla \mathbf{w})\,dx}_{A_3'} \\
 &\quad  +\veps^2\int_{\TT}\textrm{tr}([\nabla S\odot\nabla R+\nabla R\odot\nabla S]\nabla \mathbf{w})\,dx
         +\underbrace{\veps^2\xi\int_{\TT}\mathrm{tr}\Big(\big[(\Delta{S}) S+S\Delta{S}+\frac23\Delta{S}\big]\nabla\mathbf{w}\Big)\,dx}_{A_4'}
 \\
 &\quad +\underbrace{\veps^2\xi\int_{\TT}\mathrm{tr}\Big([(\Delta{S})R+R\Delta{S}]\nabla\mathbf{w}\Big)\,dx}_{A_5'}
        +\veps^2\xi\int_{\TT}\mathrm{tr}\Big([(\Delta{R})S+S\Delta{R}]\nabla\mathbf{w}\Big)\,dx \\
 &\quad -2\veps^2\xi\int_{\TT}\mathrm{tr}\Big(\big[S\,\textrm{tr}(S\Delta{R})+R\,\textrm{tr}(S\Delta{R})+S\,\textrm{tr}(R\Delta{R})\big]\nabla\mathbf{w}\Big)\,dx \\
 &\quad -\underbrace{2\veps^2\xi\int_{\TT}\mathrm{tr}\Big(\big[S\,\textrm{tr}(S\Delta{S})+R\,\textrm{tr}(R\Delta{S})+S\,\textrm{tr}(R\Delta{S})
         +R\,\textrm{tr}(S\Delta{S})\big]\nabla\mathbf{w}\Big)\,dx}_{A_6'}\\
 &\quad  -\veps a \int_{\TT}\mathrm{tr}(S^2)\,dx
         +\veps b \int_{\TT}\textrm{tr}\big(S QS+R S^2\big)\,dx
         -\veps c\int_{\TT}\textrm{tr}(Q^2)\textrm{tr}(S^2)\,dx \\
  &\quad -\veps c\int_{\TT}\textrm{tr}(R S)\textrm{tr}(QS+S R)\,dx
         +\veps\xi\int_{\TT}\textrm{tr}\Big(\big[D_{\mathbf{w}}S+SD_{\mathbf{w}}+\frac23D_{\mathbf{w}}\big]S\Big)\,dx\\
 &\quad  +\veps\xi\int_{\TT}\textrm{tr}\big([D_{\mathbf{v}}S+SD_{\mathbf{v}}+D_{\mathbf{w}}R+RD_{\mathbf{w}}]S\big)\,dx \\
 &\quad  -\underbrace{\frac23\veps\xi\int_{\TT}\big[\textrm{tr}(S\nabla\mathbf{w})+\textrm{tr}(S\nabla\mathbf{v})+\textrm{tr}(R\nabla\mathbf{w})\big]\mathrm{tr}(S)\,dx}_{I_7}       \\
 &\quad -2\veps\xi\int_{\TT}\textrm{tr}\big(\big[S\,\textrm{tr}(S\nabla{\mathbf{w}})+S\,\textrm{tr}(R\nabla{\mathbf{w}})
         +R\,\textrm{tr}(S\nabla{\mathbf{w}})+R\,\textrm{tr}(R\nabla{\mathbf{w}})\big]S\big)\,dx \\
 &\quad  -2\veps\xi\int_{\TT}\textrm{tr}\Big(\big[S\,\textrm{tr}(S\nabla{\mathbf{v}})+S\,\textrm{tr}(R\nabla{\mathbf{v}})+R\,\textrm{tr}(S\nabla{\mathbf{v}})\big]S\Big)\,dx \\
 &\quad  +\veps^2 a \int_{\TT}|\nabla S|^2\,dx
         -\veps^2 b \int_{\TT}\textrm{tr}\left([S Q +R S]\Delta S\right)\,dx
         +\veps^2 c \int_{\TT}\textrm{tr}\big(S\Delta S\big)\textrm{tr}(Q^2)\,dx\\
 &\quad +\veps^2 c\int_{\TT}\textrm{tr}(R\Delta S)\textrm{tr}(QS+S R)\,dx-\veps\int_{\TT}\textrm{tr}(\mathcal{G}(Q) \nabla\mathbf{w})\, dx\\
 &\quad +\veps\int_{\TT}\mathcal{F}_\veps(\mathbf{v},R)\cdot \mathbf{w}\, dx+\veps\int_{\T^2}\textrm{tr}(\Delta R(-\eps^2\Delta S+\eps S))\,dx.
 \end{align*}
Again using \cite[Lemma 2.1]{ADL15}, the incompressibility condition and the fact $\mathbf{v}, \mathbf{w}\in \mathcal{S}^{(3)}_0$, one can verify that
$$
  A_1=A_1',\quad A_2=A_2',\quad A_3=A_3', \quad A_4=A_4',\quad A_5=A_5',\quad A_6=A_6',
$$
as well as
$$
 I_1=I_2=I_3=I_4=I_5=I_6=I_7=0.
$$
Then using the H\"older inequality and Young' inequality, we can estimate the remaining terms on the right-hand side as follows:
\begin{align}
&\frac12 \frac{d}{dt}\int_{\TT}\Big(|\mathbf{w}|^2+\veps^2|\nabla
S|^2 +\veps|S|^2\Big)\,dx
+\int_{\TT}\Big(\veps |\nabla \mathbf{w}|^2 +\veps^3|\Delta S|^2+ \veps^2|\nabla S|^2\Big)\,dx\non\\
&\le \frac{\veps^3}{16}\|\Delta S\|_{L^2}^2
           +C\veps \|\nabla R\|_{L^\infty}^2\|\mathbf{w}\|_{L^2}^2
           +\veps^2 \|\nabla \mathbf{v}\|_{L^\infty} \|\nabla S\|_{L^2}^2
           +C\veps \|\nabla \mathbf{v}\|_{L^\infty}^2\|S\|_{L^2}^2\non\\
&\quad +\frac{\veps^3}{16}\|\Delta{S}\|_{L^2}^2
       +C\veps\xi^2\|\nabla\mathbf{v}\|_{L^\infty}^2\|S\|_{L^2}^2
       +\frac{\veps^3}{16}\|\Delta{S}\|_{L^2}^2
       +C\veps\xi^2\|\nabla \mathbf{v}\|_{L^\infty}^2(\|Q\|_{L^\infty}^2+\|R\|_{L^\infty}^2)\|S\|_{L^2}^2 \non\\
&\quad  +\|\mathbf{w}\|_{L^2}^2 +\veps^2\|\nabla R\|_{L^\infty}^2\|S\|_{L^2}^2
        +\frac{\veps}{16}\|\nabla \mathbf{w}\|_{L^2}^2
        +\veps \|R\|_{L^\infty}^2\|S\|_{L^2}^2
        +\|\nabla \mathbf{v}\|_{L^\infty}\|\mathbf{w}\|_{L^2}^2\non\\
&\quad  +\frac{\veps}{16}\|\nabla \mathbf{w}\|_{L^2}^2
        + C\veps^3 \|\Delta R\|_{L^\infty}^2 \|S\|_{L^2}^2+  C\veps^3  \|\nabla R\|_{L^\infty}^2 \|\nabla S\|_{L^2}^2 +\frac{\veps}{16}\|\nabla\mathbf{w}\|_{L^2}^2\non\\
&\quad  +C\veps^3\xi^2\|\Delta{R}\|_{L^\infty}^2\|S\|_{L^2}^2
        +\frac{\veps}{16}\|\nabla\mathbf{w}\|_{L^2}^2
        +C\veps^3\xi^2(\|Q\|_{L^\infty}^2+\|R\|_{L^\infty}^2)\|\Delta{R}\|_{L^\infty}^2\|S\|_{L^2}^2       \non\\
&\quad  +C\veps(1+\|Q\|_{L^\infty}+\|R\|_{L^\infty}+\|Q\|^2_{L^\infty}+\|R\|_{L^\infty}\|Q\|_{L^\infty}+\|R\|^2_{L^\infty})\|S\|_{L^2}^2\non\\
&\quad  +\frac{\veps}{16}\|\nabla\mathbf{w}\|_{L^2}^2
        +C\veps\xi^2(1+\|Q\|_{L^\infty}^2+\|R\|_{L^\infty}^2)\|S\|_{L^2}^2
        +C\veps\xi\|\nabla\mathbf{v}\|_{L^\infty}\|S\|_{L^2}^2\non\\
&\quad
      +\frac{\veps}{16}\|\nabla\mathbf{w}\|_{L^2}^2
      +C\veps\xi\|R\|_{L^\infty}^2\|S\|_{L^2}^2
      +\frac{\veps}{16}\|\nabla\mathbf{w}\|_{L^2}^2
      +C\veps\xi^2(\|Q\|_{L^\infty}^4+\|R\|_{L^\infty}^4)\|S\|_{L^2}^2
      \non\\
&\quad
     +C\veps|\xi|\|\nabla\mathbf{v}\|_{L^\infty}^2\|S\|_{L^2}^2
     +C\veps|\xi|(\|Q\|_{L^\infty}^2+\|R\|_{L^\infty}^2)\|S\|_{L^2}^2
     +|a|\eps^2\|\nabla S\|_{L^2}^2  \non\\
&\quad +\frac{\veps^3}{16}\|\Delta S\|_{L^2}^2
       +C\veps(\|Q\|_{L^\infty}^2+\|R\|_{L^\infty}^2+\|Q\|_{L^\infty}^4+\|R\|_{L^\infty}^4)\|S\|_{L^2}^2\non\\
&\quad + \frac{\veps}{16}\|\nabla\mathbf{w}\|_{L^2}^2+ C\eps \|\mathcal{G}(Q)\|_{L^2}^2
       +\veps^2\|\mathcal{F}_\veps(\mathbf{v},R)\|_{L^2}^2+\|\mathbf{w}\|_{L^2}^2\non\\
       &\quad +\frac{\veps^3}{16}\|\Delta S\|_{L^2}^2+\eps\|S\|_{L^2}^2+C\eps^3\|\Delta R\|_{L^2}^2,
\label{diffx}
\end{align}
where the constant $C$ is independent of $\eps$ and $\xi$.

Let
\begin{align}
&\mathcal{Y}_1(t)=\|\mathbf{w}\|_{L^2}^2+\veps^2\|\nabla S\|_{L^2}^2 +\veps\|S\|_{L^2}^2,\non\\
&\mathcal{H}_1(t)= 1+ \|\nabla \mathbf{v}\|_{L^\infty}^4
+(\veps+\veps^2)(\|\nabla R\|_{L^\infty}^2+\|\Delta R\|_{L^\infty}^2+\|\Delta R\|_{L^\infty}^4) \non\\
&\qquad \quad +(1+\veps^2)(\|R\|_{L^\infty}^4 +\|Q\|^4_{L^\infty}).\non
\end{align}
From the definition of $\mathcal{G}(Q)$, it easily follows that $\|\mathcal{G}(Q)\|_{L^2}\leq C(1+\|Q\|^4_{L^\infty})$ with some $C$ independent of $\eps$.
Then using the Cauchy--Schwarz inequality, we infer from \eqref{diffx} that
\begin{align}
&\frac{d}{dt}\mathcal{Y}_1(t)+\veps \|\nabla \mathbf{w}\|_{L^2}^2 +\veps^3\|\Delta S\|_{L^2}^2+ \veps^2\|\nabla S\|_{L^2}^2\non\\
&\quad \leq
C\mathcal{H}_1(t)\mathcal{Y}_1(t)+C\veps(1+\|Q\|_{L^\infty}^4)^2+\eps^2\|\mathcal{F}_\veps(\mathbf{v},R)\|_{L^2}^2+ C\eps^3\|\Delta R\|_{L^2}^2,\non
\end{align}
where $C$ is a constant that may depend on $\xi$ but is independent of $\veps$.

Since $\mathcal{H}_1(t)$,  $\|\mathcal{F}_\veps(\mathbf{v},R)\|_{L^2}$ are a priori bounded (in particular $\|Q\|_{L^\infty}$ is now assumed to be bounded), from the fact $\mathcal{Y}_1(0)=0$ and Gronwall's lemma, we deduce that for $0< \veps< 1$, it holds
\begin{align}
&\mathcal{Y}_1(t) \leq C_T \veps^2, \quad \forall\, t\in [0,T],\non\\
& \int_0^T\left(\veps \|\nabla \mathbf{w}\|_{L^2}^2 +\veps^3\|\Delta S\|_{L^2}^2+ \veps^2\|\nabla S\|_{L^2}^2\right)dt\leq C_T \veps^2.\non
\end{align}

The proof is complete.
\end{proof}

\subsection{Proof of Theorem~\ref{thm:lackofeigenv}}

We are in a position to prove Theorem~\ref{thm:lackofeigenv}. This can be done by a contradiction argument.
To this end, we assume that if
\begin{equation}\label{initial setting+}
\mathrm{\lambda}_i(Q_0(x))\in\left[-\dfrac{b+\sqrt{b^2-24ac}}{12c},\ \ \dfrac{b+\sqrt{b^2-24ac}}{6c}\right],
\qquad \forall\, x\in\TT,\ \ 1\leq i\leq 3,
\end{equation}
then for\emph{ any} $\xi\not=0$,  $\veps>0$, $a\in\R$, $b> 0$, $c>0$ and any $t>0$, we still have the eigenvalue-range preservation:
 \begin{equation}\label{initial setting+a}
\mathrm{\lambda}_i(Q^\eps (x,t))\in\left[-\dfrac{b+\sqrt{b^2-24ac}}{12c},\ \ \dfrac{b+\sqrt{b^2-24ac}}{6c}\right],
\qquad \forall\, x\in\TT,\ \ 1\leq i\leq 3.
\end{equation}

Since the physicality of a
$Q$-tensor imposes an upper bound on the size of $\|Q\|_{L^\infty}$, then thanks to the bound on eigenvalues \eqref{initial setting+a}, we see that \eqref{ass:linftybdeps} is satisfied with certain $\eta>0$. As a consequence, from Proposition~\ref{thm:highEricksen2} we infer that the limit system as $\veps\to 0$ i.e., \eqref{nondim:v10x}--\eqref{nondim:inix},  also has the eigenvalues in the same range (because pointwise convergence of $Q$-tensors implies pointwise convergence for their eigenvalues, as eigenvalues are continuous functions of matrices, see for instance \cite{nomizu}).

Next, we recall the $R$ equation
\begin{align}
&\delt R+\mathbf{v}\cdot\nabla R=(\xi D_{\mathbf{v}}+ \Omega_{\mathbf{v}})\big( R+\frac13 \Id
\big)+\big(R+\frac13 \Id \big)(\xi
D_{\mathbf{v}}- \Omega_{\mathbf{v}}) -2\xi\big(R+\frac13 \Id \big)\tr(R\nabla \mathbf{v})\non\\
&\qquad \qquad \qquad \qquad -aR+b\big[R^2-\frac{1}{3}\tr(R^2)\Id\big]- cR\tr(R^2),
\end{align}
where $\mathbf{v}$ is a solution of the Euler equation \eqref{nondim:v10x}. Since $\mathbf{u}_0\in H^{5}$ and $Q_0\in H^4$, we know that $\mathbf{v}\in L^\infty(0,T;H^{5}(\TT;\RR^3))$ and $R\in L^\infty(0,T; H^{4}(\TT;\sS^{(3)}_0))$.
Now we define
$$\mathbf{v}^\xi(x,t):=\mathbf{v}\Big(x, \frac{t}{\xi}\Big),\quad \Rx (x,t):=R\Big(x,\frac{t}{\xi}\Big), $$
and then the pair
$(\mathbf{v}^\xi, R^\xi)$ satisfies the following system
\begin{align}
&\delt \mathbf{v}^\xi+\frac{1}{\xi}\mathbf{v}^\xi\cdot\nabla \mathbf{v}^\xi+\frac{1}{\xi}\nabla q(x, \frac{t}{\xi})=0,\non\\
&\nabla\cdot \mathbf{v}^\xi=0,\non\\
&\delt \Rx+ \frac{1}{\xi}\mathbf{v}^\xi\cdot\nabla \Rx
=(D_\mathbf{\vx}+ \frac{1}{\xi}\Omega_\mathbf{\vx})\big( \Rx+\frac13 \Id
\big)+\big(\Rx+\frac13 \Id \big)(
D_\mathbf{\vx}- \frac{1}{\xi}\Omega_\mathbf{\vx})-2\big( \Rx+\frac13 \Id \big)\tr(\Rx \nabla\mathbf{\vx})\non\\
&\qquad \qquad\qquad\qquad\quad
-\frac{1}{\xi}\Big( a\Rx-b\big[(\Rx)^2-\frac{1}{3}\mathrm{tr}((\Rx)^2)\Id\big]+ c\Rx\mathrm{tr}((\Rx)^2)\Big).\non
\end{align}
Applying the curl operator to the first equation, we see that $\omega_\xi:=\mathrm{curl} \mathbf{v}^\xi$ satisfies
\be
\delt \omega_\xi+\frac{1}{\xi}\mathbf{v}^\xi\cdot\nabla\omega_\xi=\frac{1}{\xi}\omega_\xi\cdot \nabla \mathbf{v}^\xi.\non
\ee
Since $\mathbf{v}\in L^\infty(0,T;H^{5}(\TT; \RR^3))$, we infer that there exists certain constant $C>0$ independent of $\xi$ such that  $\|\omega_\xi\|_{L^\infty(0,T; H^{4}(\TT;\R^3))}\le C$, provided that $\xi\ge 1$.
Hence, we deduce that as $\xi\to\infty$
\be
\mathbf{v}^\xi\to \mathbf{u}_0\,\textrm{ in } L^\infty(0,T;H^4(\TT)),\non
\ee
where $\mathbf{u}_0$ is the initial data for the velocity field.

Moreover, we have  $\Rx (x, \cdot)\to \Ro(x,t)$ as $\xi\to\infty$ uniformly for $t\in [0,T]$ and $x\in \TT$,
where $\Ro(x,\cdot)$ solves the ODE system (parametrized by $x$):
\be
\delt \Ro = D_{\mathbf{u}_0}\Ro+\Ro
D_{\mathbf{u}_0}+\frac{2}{3}D_{\mathbf{u}_0}-2\big( \Ro+\frac13 \Id \big)\tr(\Ro D_{\mathbf{u}_0}),\non
\ee
where $\displaystyle{D_{\mathbf{u}_0}:=\frac{\nabla \mathbf{u}_0+\nabla^T \mathbf{u}_0}{2}}$.
If we take the initial datum $R_0|_{t=0}=0$, we will have, as long as the matrix $D_{\mathbf{u}_0}\not\equiv 0$ then the solution $R_0(t)\not\equiv 0$, namely,
 there exists  some $t_0\in [0,T]$ such that $R_0(t_0)\not=0$.

As a consequence,  there exists a  sufficiently small $\lambda>0$ (depending on this nonzero $R_0(t_0)$) such that if we take the coefficients
$$a=\lambda a_0,\quad b=\sqrt{\lambda}b_0> 0,\quad  c=c_0>0$$
that  satisfy
$$|\lambda a_0|<  \frac{\lambda b_0^2}{3c_0},$$
we have
$$\mathrm{\lambda}_i(R_0(0))=0\in\left[-\dfrac{\sqrt{\lambda} b_0+\sqrt{\lambda b_0^2-24\lambda a_0c_0}}{12c_0},\ \ \dfrac{\sqrt{\lambda} b_0+\sqrt{\lambda b_0^2-24\lambda a_0c_0}}{6c_0}\right],$$
but for  $t_0\in [0,T]$, it holds
$$\mathrm{\lambda}_i(R_0(t_0))\not\in\left[-\dfrac{\sqrt{\lambda} b_0+\sqrt{\lambda b_0^2-24\lambda a_0c_0}}{12c_0},\ \ \dfrac{\sqrt{\lambda} b_0+\sqrt{\lambda b_0^2-24\lambda a_0c_0}}{6c_0}\right].$$
This leads to a contradiction, which proves our assertion.

The proof of Theorem~\ref{thm:lackofeigenv} is complete.\hfill$\Box$

\section{Dynamical emergence of defects}
\label{sec:defectsdyn}
\setcounter{equation}{0}
In this section, we investigate some qualitative features of solutions to the system (i.e., the ``limit system" of \eqref{nondim:u1}--\eqref{nondim:Q1} with $\xi=0$):
\begin{align}
& \delt \mathbf{v}+ \mathbf{v}\cdot\nabla\mathbf{v}+\nabla q =0,\label{eq:limEuler}\\
&\nabla\cdot \mathbf{v} =0,\\
& \delt R+ \mathbf{v}\cdot\nabla R-\Omega_\mathbf{v} R+ R\Omega_\mathbf{v} =-aR+b\big[R^2-\frac{1}{3}\tr(R^2)\Id\big]- cR\tr(R^2),
\label{eq:limitTrans}
\end{align}
Our aim is to understand how it can describe liquid crystal defects that are understood as high gradients. We shall provide some specific examples of flows providing mechanisms responsible for this phenomenon, which can be regarded as the generation of defect patterns.

As has been seen in the previous sections, the equation for $R$ is well-posed in the spaces in which we work, so if one starts with
 $(\mathbf{v}_0,R_0)\in H^{5}(\TT)\times H^{4}(\TT)$ then $(\mathbf{v},R)(\cdot,t)\in H^{5}(\TT)\times H^{4}(\TT)$ for
 $t>0$ (see for instance, Proposition~\ref{prop:globalexistlimitsyst}  in the Appendix).
 However, what can happen is that for certain types of initial data the $L^\infty$-norm of the gradients of $R$ can increase in time on the interval $[0,T]$.
Below we present two distinct ways of generating defects:
\begin{itemize}
\item {\it  phases mismatch};
\item {\it vorticity-driven defects}.
\end{itemize}

\subsection{\bf The phases mismatch}

We first consider the case when the flow is not present, namely, $\mathbf{v}\equiv 0$.
Then the equation \eqref{eq:limitTrans} reduces to the gradient flow of an ODE system:
\be\label{ode:gradflow}
\frac{d}{dt}R=-aR+b\big[R^2-\frac{1}{3}\tr(R^2)\Id\big]- cR\tr(R^2),
\ee
whose solutions evolve in time towards one of the two stable minima, namely, the local minimum $\left\{s_-\left(\mathbf{n}\otimes \mathbf{n}-\frac{1}{3}\Id\right),\ \mathbf{n}\in\mathbb{S}^2\right\}$, respectively the global minimum $\left\{s_+(\mathbf{n}\otimes \mathbf{n}-\frac{1}{3}\Id),\ \mathbf{n}\in\mathbb{S}^2\right\}$.
Here,
$$s_-=\frac{b-\sqrt{b^2-24ac}}{4c}<0<s_+=\frac{b+\sqrt{b^2-24ac}}{4c},$$
with  $a<0$, $b,c>0$. 
More specifically, one can easily check that if we take for instance
$$
R_0(x_1,x_2)=x_1e^{-|x_1|^2}\left(e_1\otimes e_1-\frac{1}{3}\Id\right),\, \quad  x\in\TT,
$$
where $e_1=(1,0,0)$,
then the solution of \eqref{eq:limitTrans} will be $R(x_1,x_2,t):=s(x,t)\left(e_1\otimes e_1-\frac{1}{3}\Id\right)$ with
\begin{equation}
s(x,t)\to
\left\{
\begin{aligned}
s_-, &\quad  \textrm{ if }x_1<0,\\
s_+, &\quad  \textrm{ if }x_1>0,
\end{aligned}
\right.
\quad
\textrm{ as } t\to\infty.\non
\end{equation}
Therefore, as $t\to\infty$ the gradient of $R$ at the plane $x_1=0$ will increase towards infinity as $t\to\infty$.

\subsection{\bf Vorticity-driven defects}

Given the roughness specific to Euler equation, one can expect that defects are related to discontinuities and high gradients in the flow.
However, this is not the only possible mechanism. Indeed, we will show that even for a very well-behaved flow, e.g.,  a stationary flow we are able to get for generic points that the gradients of the $Q$-tensor will increase in time.

Below we  consider two cases that are in some sense extreme: (1) a very general flow, but very specific initial data, and  (2) a very general initial datum, but with more restrictions on the flow.

\subsubsection{\bf General flow and special initial data}
We have the following:

\begin{proposition}\label{prop:specdata}
Let $\mathbf{v}:\T^2\to \R^2$  be a $C^2$ stationary solution of the incompressible Euler system in $2D$.
We denote by $\omega$ its vorticity, namely $\omega:=\partial_x v_{2}-\partial_y v_{1}$ and assume that there exists a sequence $\{(x_k,y_k)\}_{k\in \N}\subset\TT$ with $(x_k,y_k)\to (\bar x,\bar y)\in\TT$ as $k\to\infty$ and $\omega(x_k,y_k)-\omega(\bar x,\bar y)\not=0$, for all $k\in \N$.
We further assume that $a<0$, $b,c>0$ and let $Q_0(x,y)\displaystyle{=s_+\Big(e_1\otimes e_1-\frac{1}{3}\Id\Big)}$, with $s_+=\displaystyle{\frac{b+\sqrt{b^2-24ac}}{4c}}$ and $e_1=(1,0,0)$.

Then, denoting $\mathbf{v}:=(v_1,v_2,0)(x,y)$, and letting $Q_0$ be an initial datum for equation \eqref{eq:limitTrans} with the given flow $\mathbf{v}$,  we see that there exists a sequence of times $t_k\to\infty$ and points $(\tilde x_k, \tilde y_k)$ (with $(\tilde x_k,\tilde y_k)$ in the segment connecting $(x_k,y_k)$ and $(\bar x, \bar y)$) such that
 $$|\nabla R(\tilde x_k,\tilde y_k, t_k)|\to\infty,\quad \text{as }k\to \infty.$$
\end{proposition}

\begin{remark}
For all non-degenerate critical points $(\bar x, \bar y)$ of the vorticity $\omega$ one can choose an approximating sequence $\{(\tilde x_k, \tilde y_k)\}_{k\in\N}$ in a subset $V\subset\TT$ such that $\omega|_V$ is a non-degenerate local extremum, hence the assumption  $\omega(\tilde x_k, \tilde y_k)-\omega(\bar x, \bar y)\not=0$, for all $k\in \N$ is satisfied.
\end{remark}

\begin{proof}
We consider a Lagrangian perspective.  Recall  that for the given flow $\mathbf{v}$,
the corresponding particle-trajectory map  $X:\TT\times\R\to\TT$  is defined  as the solution of the following ODE system:
\be\label{def:ODEflow}
\left\{\begin{aligned}
& \frac{dX}{dt}(\alpha,t)=\mathbf{v}(X(\alpha,t))\qquad \textrm{ mod }2\pi,\\
& X(\alpha,0)=\alpha\in \TT.
\end{aligned}
\right.
\ee
Since the torus is a compact manifold, the solution exists for $t\in (-\infty,+\infty)$.

We can thus now rewrite \eqref{eq:limitTrans} as an ODE along particle paths:
\be\label{ode:paths-1}
\Big(\frac{d}{dt}R-\Omega R+R\Omega\Big)(X(\alpha,t),t)=-\frac{\partial f_B(R)}{\partial R}(X(\alpha,t),t),
\ee
 where  the matrix $\Omega$ is related to the vorticity of the fluid, namely
\be
\Omega:=\left(\begin{array}{lll} 0 & \omega & 0 \\
-\omega & 0 &0 \\
0 &  0  &0 \end{array}\right),\qquad \text{with}\ \omega:=\partial_x u_{2}-\partial_y u_{1}.
\ee
Furthermore, because the flow is $2D$, the vorticity is constant along particle paths (see for instance, \cite[Corollary 1.2]{Majdabook}):
\be\label{2dvorticitytran}
\omega(X(\alpha,t))=\omega_0(\alpha).
\ee
In order to further simplify the equation \eqref{ode:paths-1}, we introduce the
operator $B:\TT\times \R\to \mathbb{M}^{3\times 3}$ as the solution of the following
system:
\be\label{eq:B}
\left\{\begin{aligned}
&\frac{d}{dt}B(\alpha,t)=\Omega(X(\alpha,t))B(\alpha,t),\\
&B(\alpha,0)=\Id.
\end{aligned}
\right.
\ee
\begin{lemma} $B$ is a rotation operator, i.e.,
 $B(\alpha,t)\in O(3)$, $\forall\,\alpha\in\TT$ and $t\ge
0$.
\end{lemma}
\begin{proof}
Note that
$$\frac{d}{dt}B^T(\alpha,t)=-B^T(\alpha,t)\Omega(X(\alpha,t)).$$
 Let $M(\alpha,t):=B(\alpha,t)B^T(\alpha,t)$. We see that $M(\alpha,t)$ is a solution of the ODE system
\begin{equation}
\left\{
\begin{aligned}
&\frac{d}{dt}M(\alpha,t)=\Omega(X(\alpha,t)) M(\alpha,t)-M(\alpha,t)\Omega(X(\alpha,t)),\\
&M(\alpha,0)=\Id.
\end{aligned}
\right.\non
\end{equation}
It is obvious that  $\Id$ is also a solution of this linear system, then by uniqueness we have
$$
 M(\alpha,t)=\Id,\quad \forall\, \alpha\in\TT,\ t\ge 0,
$$
which yields the conclusion.
\end{proof}
Thanks to the rotation operator $B:\TT\times \R\to O(3)$, we can
further derive
\begin{lemma}
$$
R(X(\alpha,t),t)=B(\alpha,t)Q_0B^T(\alpha,t),\quad \forall\, \alpha\in\TT,\ t\ge 0.
$$
\end{lemma}
\begin{proof}
Denote
$$
U(\alpha,t):=B^T(\alpha,t)R(X(\alpha,t))B(\alpha,t).
$$
Then the system \eqref{ode:paths-1} becomes
\begin{align}
\frac{d}{dt}U(\alpha,t)
&=(\frac{d}{dt}B^T)RB+B^T(\frac{d}{dt}R)B+B^TR(\frac{d}{dt}B)\non\\
&=\left(B^T\left(\frac{d}{dt}R-\Omega R+R\Omega\right)B\right)(X(\alpha,t),t)\non\\
&=\left(B^T\left(-aR+b (R^2-\frac{1}{3}\tr(R^2)\Id)-cR\tr(R^2)\right)B\right)(X(\alpha,t),t)\non\\
&=\left(-a U+b \Big(U^2-\frac{1}{3}\tr(U^2)\Id\Big)-c U\tr(U^2)\right)(\alpha,t).
\label{eq:paths}
\end{align}
We know that the equation \eqref{eq:paths}  satisfied by  $U$ is a gradient flow, hence except for the case when one start with zero initial data, in the long-time we will have evolution to one of the two manifolds of steady states, namely,
$$\left\{s_+\left(\mathbf{n}\otimes \mathbf{n}-\frac{1}{3}\Id\right),\ \mathbf{n}\in\mathbb{S}^{2}\right\}\bigcup\left\{s_-\left(\mathbf{n}\otimes \mathbf{n}-\frac{1}{3}\Id\right),\ \mathbf{n}\in\mathbb{S}^{2}\right\},\quad\text{with } s_\pm=\frac{b\pm\sqrt{b^2-24ac}}{4c}.$$
 Moreover, taking the initial data in one of these two manifolds, as we do by the initial assumptions, we have that the solution is stationary in time, hence $U(\alpha,t)=U(\alpha,0)$, which in terms of $R$ becomes:
$$
R(X(\alpha,t),t)=B(\alpha,t)Q_0B^T(\alpha,t).
$$
The proof is complete.
\end{proof}
 Because the flow is defined for all real times, in terms of $R$ we
 infer from the above lemma that
\begin{align}
R(\alpha,t)&=s_+B(X(\alpha,-t),t)\left(e_1\otimes e_1-\frac{1}{3}\Id\right)B^T(X(\alpha,-t),t)\non\\
&=s_+\left(\mathbf{m}(\alpha,t)\otimes \mathbf{m}(\alpha,t)-\frac{1}{3}\Id\right),\non
\end{align}
 where
\be\label{eq:malphat0}
\mathbf{m}(\alpha,t)=B(X(\alpha,-t),t)e_1.
\ee
On the other hand, taking into account the definition \eqref{eq:B} of $B(\alpha,t)$,  we have
\be
B(\alpha,t)=e^{\Omega(\alpha,t)}=\left(\begin{array}{lll} \cos(\pi t\omega(\alpha))  & \sin(\pi t\omega(\alpha))  & 0 \\ -\sin(\pi t\omega(\alpha)) & \cos(\pi t\omega(\alpha))  & 0 \\ 0 & 0 & 0 \end{array}\right),\non
\ee
where for the last equality we used the formula for the exponential of a matrix (see for instance, \cite{Hartmanbook}).
Using  the fact that in $2D$ the vorticity is conserved along the flow map \eqref{2dvorticitytran},  we have
 $$B(X(\alpha,-t),t)=B(\alpha,t).$$
 Hence, \eqref{eq:malphat0} becomes
\be
\mathbf{m}(\alpha,t)=\left(\cos\left(\pi t\omega(X(\alpha,-t))\right), \  -\sin\left(\pi t\omega(X(\alpha,-t))\right),\ 0\right).\non
\ee

Taking $\alpha=(x_k,y_k)$ and  $\alpha=(\bar x, \bar y)$, respectively,  we can calculate the angle between the two vectors $\mathbf{m}$ at these points, at times  $t_k$ (to be fixed later) as
\be
\mathbf{m}(x_k,t_k,t_k)\cdot \mathbf{m}(\bar x,\bar y ,t_k)=\cos \gamma_k,\quad \text{with}\ \ \gamma_k:=\pi t_k\big(\omega(x_k,y_k)-\omega(\bar x,\bar y)\big).\non
\ee
Recalling now the assumption in the hypothesis that $\omega(x_k,y_k)-\omega(\bar x,\bar y)\not=0$, we can take the times
$$
t_k:=\Big(2\big[\omega(x_k,y_k)-\omega(\bar x,\bar y)\big]\Big)^{-1}\to \infty.
$$
As a consequence, we have that as $(x_k,y_k)\to (\bar x,\bar y)$ the angle between $R(x_k,y_k)$ and $R(\bar x,\bar y)$ stays non-zero and fixed, as $t_k\to\infty$.
This leads to the final claim of the proposition.
\end{proof}

\subsubsection{General initial data and special points in the flow}

In dimension two, a divergence-free flow can be expressed in terms of a stream function. More precisely, if $\mathbf{v}:\TT\to\R^2$ is smooth enough with $\nabla\cdot \mathbf{v}=0$ and moreover $\int_{\TT} \mathbf{v}\,dx=0$ then  there exists a scalar function $\psi:\TT\to\R$, called the stream function, such that $\mathbf{v}=(\partial_y \psi, -\partial_x \psi)$.
Thus the system  providing the particle-trajectory map \eqref{def:ODEflow} becomes a Hamiltonian flow, and the stagnation points of the flow are given as critical points of the stream function $\psi$. Locally, the flow around them will be as in Figure~1 (see for instance, \cite[Chapter 1]{GHbook}).

\begin{figure}[h]\label{localflowsfig}
\includegraphics[scale=0.6]{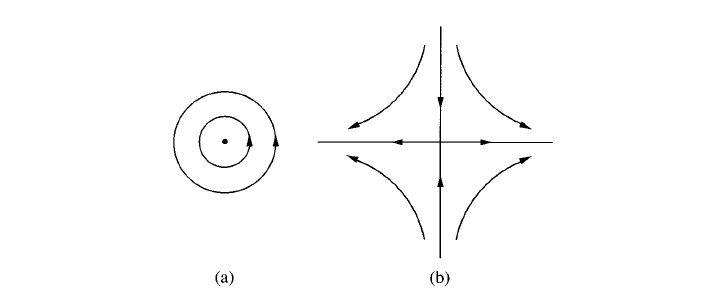}
\caption{Two kinds of stagnation points of the flow: (a) an elliptic fixed point corresponding
to a local minimum or maximum of $\psi$, and (b) a hyperbolic fixed point at a saddle point of $\psi$}
\end{figure}

We will need to assume essentially that around such points the vorticity is not constant, namely we have the kind of assumption of non-degeneracy imposed in Proposition~\ref{prop:specdata}. We expect that  this kind of assumption holds generically, so we will call these ``standard stagnation points". More precisely,

\begin{definition}[Standard stagnation points in incompressible flows]
\label{def:stanstag}
Consider an incompressible flow $\mathbf{v}\in C^2(\TT;\R^2)$ with  zero mean. We call $(\bar x, \bar y)\in \TT$ a
{\rm standard stagnation point for the flow $\mathbf{v}$}, if $\mathbf{v}(\bar x, \bar y)=\mathbf{0}$ and there
exists a sequence of points $\{(x_k, y_k)\}_{k\in N}\subset\TT$ with $(x_k, y_k)\to (\bar x, \bar y)$  such that $|X(x_k,y_k,t)-X(\bar x,\bar y,t)|\to 0$, \textrm{uniformly for} $t\in\R_+
$ and $\omega(x_k, y_k)-\omega(\bar x, \bar y)\not=0$ for any $k\in\N$ (where $X:\TT\times\R\to\TT$ is the
 particle-trajectory map generated by the flow $\mathbf{v}$ and $\omega:=\partial_x v_{2}-\partial_y v_{1}$ is the vorticity of the flow).
\end{definition}

Then we have the following:

\begin{proposition}
Let $\mathbf{v}:\TT\to \R^2$ be a $C^2$ stationary solution of the incompressible Euler system with zero mean.
Let  $(\bar x, \bar y)$ be a standard stagnation point in the sense of Definition~\ref{def:stanstag} and $Q_0:\TT\to \sS_0^{(3)}$ be a $C^2$ function such that $Q_0(\bar x, \bar y)$ is a biaxial $Q$-tensor that belongs to a suitable neighborhood $\mathcal{U}$ of the set $\displaystyle{\Big\{s_+\left(\mathbf{n}\otimes \mathbf{n}-\frac{1}{3}\Id\right), \mathbf{n}\in\mathbb{S}^2\Big\}}$ (as defined in Proposition~\ref{prop:localODE} in the Appendix).

Then there exist a sequence of points $\{(x_k,y_k)\}_{k\in\N}$ converging to $(\bar x,\bar y)$ and a sequences of times $t_k\to\infty$  such that
$$|\nabla R(x_k,y_k, t_k)|\to\infty,\quad \text{as }k\to \infty.$$
\end{proposition}
\begin{proof}
We consider, similarly as in the proof of Proposition~\ref{prop:specdata}, the particle-trajectory map $X:\TT\times \R\to \TT$, generated by the flow $\mathbf{v}$. Then following the reasoning in the proof of Proposition~\ref{prop:specdata} we obtain:

\be
|R(X(\alpha,t))-B(\alpha,t)\mathcal{T}[Q_0(\alpha)]B(\alpha,t)^T|\to 0,\quad \textrm{ as }t\to\infty,\non
\ee
where $\mathcal{T}[Q_0]=\displaystyle{s_+\Big(\mathbf{n}(Q_0)\otimes \mathbf{n}(Q_0)-\frac{1}{3}\Id\Big)}$ (with $\mathbf{n}(Q_0)\in\mathbb{S}^2$) is the long-time limit of the gradient flow \eqref{eq:paths} starting from initial data $Q_0$ (see Proposition~\ref{prop:localODE} in the Appendix).

Thanks to the assumption that $(\bar x,\bar y)$ is a standard stagnation point (see Definition~\ref{def:stanstag}),
we deduce that there exists a sequence of points $\{(x_k,y_k)\}_{k\in \N}$ converging to $(\bar x,\bar y)$ such that
\be
|X(x_k,y_k,t)-X(\bar x,\bar y,t)|\to 0,\quad  \textrm{ uniformly for }t\in\R_+\non
\ee
 and
\be\label{ass:nondegxyk}
\omega(x_k,y_k)-\omega(\bar x,\bar y)\not=0,\quad \forall\, k\in\N.
\ee
On the other hand, we have
$$B(\alpha,t)\mathcal{T}[Q_0(\alpha)]B(\alpha,t)^T=s_+\Big(B(\alpha,t)\mathbf{n}(Q_0(\alpha))\otimes B(\alpha,t)\mathbf{n}(Q_0(\alpha))-\frac{1}{3}\Id\Big).$$

On a simply-connected neighbourhood  of $\mathcal{V}$ of $(\bar x,\bar y)$ we can define the  ``angle function" $A:\mathcal{V}\to [0,2\pi]$ depending on the function $Q$ such that $\mathbf{n}(Q(x,y))=(\cos(A(x,y)),\sin(A(x,y),0)$ for all $(x,y)\in\mathcal{V}$.
Taking into account that $\mathbf{n}(Q)$ is a continuous function on the neighbourhood $\mathcal{V}$ (see Proposition~\ref{prop:localODE}
 in the Appendix), the angle function $A$ (or its lifting in topological jargon) can be chosen to be continuous.

Thus, with this notation we obtain
\begin{align}
& \Big(B(x_k,y_k,t_k)\mathcal{T}[Q_0(x_k,y_k)]B(x_k,y_k,t_k)^T\Big) : \Big(B(\bar x,\bar y,t_k)T[Q_0(\bar x,\bar y)]B(\bar x,\bar y,t_k)^T\Big)\non\\
& \quad =s_+^2 \left[ \left(B(\cdot,t_k)\mathbf{n}(Q_0)|_{(x_k,y_k)}\cdot B(\cdot,t_k)\mathbf{n}(Q_0)|_{(\bar x,\bar y)}\right)^2-\frac{1}{3}\right]\non\\
& \quad =s_+^2 \left((\cos \gamma_k)^2-\frac{1}{3}\right),\non
\end{align}
 where
\be
\gamma_k:=(A-\pi t_k\omega)(x_k,y_k)-(A-\pi t_k\omega)(\bar x,\bar y).\non
\ee
Because of \eqref{ass:nondegxyk}, we can take
$$
t_k:=\Big(2\big[\omega(x_k,y_k)-\omega(\bar x,\bar y)\big]\Big)^{-1}\to \infty.
$$
Then, using the continuity of the angle function $A$ we deduce that  $\gamma_k\to \frac{\pi}{2}$ as $k\to \infty$.

The proof is complete.
\end{proof}

\begin{appendix}
\section{The nonlinear Trotter product formula}
\label{sec:Trotter}
\setcounter{equation}{0}

Suppose $u$ is a solution to the parabolic problem
\begin{equation}
\left\{
\begin{aligned}
&\dfrac{\partial{u}}{\partial{t}}=A(t){u}+X(u),\\
&u(0)=f,
\end{aligned}
\right.
\label{parabolic equ}
\end{equation}
One can obtain its solutions by successively solving two
simpler equations
\begin{equation}
\dfrac{\partial{u}}{\partial{t}}=A(t){u}, \quad
\dfrac{\partial{u}}{\partial{t}}=X(u).
\end{equation}
Such Trotter  product formulas combining the solutions of the two simpler equations are available in the literature in the case when $A(t)$ is a time-independent operator (see for instance, \cite{T97}). However, such results do not seem to be immediately available in the literature for time-dependent operators as we need. So we provide here a brief argument showing the specific result we shall use, which follows closely the argument in \cite[Proposition 5, pp. 310]{T97}.

We start by recalling the following definition  (cf. \cite{P83}, pp. $129$)
\begin{definition}\label{clodef}
Let $V$ be a Banach space. The two-parameter family of bounded linear operators  $\mathcal{U}(t,s)$ for $0\le s\le t\le T$ on $X$    is called an {\rm evolution system} if:

(i) $\mathcal{U}(s, s)=I$,  $\mathcal{U}(t, s)=\mathcal{U}(t,\tau)U(\tau, s)$ for any $s\leq \tau\leq t$ ;

(ii) the mapping $(t,s)\to \mathcal{U}(t,s)$ is strongly continuous for $0\le s\le t\le T$.
\end{definition}

Then we denote by $\mathcal{U}(t, s)$ the evolution system generated by the
time-dependent operator $A(t)$ on the Banach space $V$. Besides, we also let $\F^t$ be the one-parameter semigroup generated on $V$ by the nonlinear ODE $\dfrac{\partial{u}}{\partial{t}}=X(u)$.

For $T>0$, $n\in \mathbb{N}$, $k\in \{1,...,n\}$, we set
\begin{equation}
\label{v equ 1}
v_k=\left(\mathcal{U}\Big(\frac{k}{n}T,\frac{k-1}{n}T\Big)\F^{\frac{T}{n}}\right)\circ\cdots\left(\mathcal{U}\Big(\frac{2T}{n},
\frac{T}{n}\Big)\F^{\frac{T}{n}} \right)\circ\left(\mathcal{U}\Big(\frac{T}{n},
0\Big)\F^{\frac{T}{n}}  \right)f,
\end{equation}
and furthermore
\begin{equation}\label{v equ 2}
v(t)=\mathcal{U}\Big(\frac{k}{n}T+\eta,\frac{k}{n}T\Big)\F^\eta{v}_k,
\quad\mbox{for } t=\frac{kT}{n}+\eta, \; 0\leq \eta<\frac{T}{n}.
\end{equation}
\begin{proposition}
\label{proposition in T}
Let $V$ and $W$ be two Banach spaces such that $V$ is continuously
embedded in $W$. We make the following assumptions.

(H1) Suppose that the evolution system $\mathcal{U}(t, s)$ generated by $A(t)$ on $V$
satisfies
\begin{equation}
\label{property of A}
\|\mathcal{U}(t, s)\|_{\LL(V)}\leq e^{C(t-s)},
\quad\|\mathcal{U}(t, s)\|_{\LL(W,V)}\leq C(t-s)^{-\gamma},
\quad\|\mathcal{U}(t, s)-I\|_{\LL(V,W)}\leq C(t-s)^{\delta},
\end{equation}
for $0\leq s<t\leq T$, with some $\delta>0$, $0<\gamma<1$. Here, $I$ stands for the identity operator and $\LL(V,W)$ denotes the set of bounded linear operators from $V$ into $W$.

(H2) Let  $\F^t$ be the flow generated by the vector field $X$. Assume it  satisfies:
\begin{align*}
&\|\F^t(f)\|_V\leq C_2, \qquad\quad\ \ \mbox{for } \|f\|_V\leq C_1,\\
&\|\F^t(f)\|_V\leq e^{Ct}\|f\|_V,\quad\, \mbox{for }\|f\|_V\geq C_1,
\end{align*}
for $0<t\leq T$. Here, $C_1, C_2, C>0$ are some constants independent of $f\in V$.

We also assume that
\begin{equation}\label{bound-X}
X: V\rightarrow V \ \mbox{and}\  \mathcal{Y}: V\times
V\rightarrow\LL(W)\cap\LL(V) \;\mbox{are bounded},
\end{equation}
where
$$
  \mathcal{Y}(h, g)\defeq\int_0^1 DX\big(sh+(1-s)g\big)\,ds.
$$

For any  $f\in V$, we let  $u\in C([0, T]; V)$ be the solution to problem \eqref{parabolic
equ}, and let $v\in C([0, T]; V)$ be defined by \eqref{v equ 1}--\eqref{v
equ 2} above. Then we have the following  error estimate:
\begin{equation}
\|v(t)-u(t)\|_V \leq C(T, C_1,C_2,\|f\|_V)n^{-\delta},\quad 0\leq t\leq T.\non
\end{equation}
\end{proposition}
\begin{proof}
The proof is adapted from that of \cite[Proposition 5.1]{T97}. We note that for $\frac{kT}{n}< t<\frac{(k+1)T}{n}$,
\begin{align*}
\frac{\partial{v}}{\partial{t}}&=A(t)v+\mathcal{U}\Big(\frac{k}{n}T+\eta,\frac{k}{n}T\Big)X(\F^\eta{v}_k)\\
&=A(t)v+X(v)+R(t), \quad \text{with}\ \eta=t-\frac{kT}{n}\in \Big(0, \frac{kT}{n}\Big),
\end{align*}
where
\begin{align*}
R(t)&\defeq
\mathcal{U}\Big(\frac{k}{n}T+\eta,\frac{k}{n}T\Big)X(\F^\eta{v}_k)-X(v)\\
&=\Big[\mathcal{U}\Big(\frac{k}{n}T+\eta,\frac{k}{n}T\Big)-I\Big]X(\F^\eta{v}_k)+\Big[X(\F^\eta{v}_k)-X(v)\Big].
\end{align*}
We infer from the assumptions (H1), (H2) the following estimates
\begin{align}
&\|X(\F^\eta{v}_k)-X(v)\|_{W}\nonumber\\
&\quad =\left\|X(\F^\eta{v}_k)-X\Big(\mathcal{U}\Big(\frac{k}{n}T+\eta,\frac{k}{n}T\Big)\F^\eta{v}_k\Big)\right\|_{W}\nonumber\\
&\quad \leq C \left\| \mathcal{Y}\Big(\F^\eta{v}_k, \mathcal{U}\Big(\frac{k}{n}T+\eta,\frac{k}{n}T\Big)\F^\eta{v}_k \Big)\right\|_{\mathcal{L}(W)}
\left\|\mathcal{U}\Big(\frac{k}{n}T+\eta,\frac{k}{n}T\Big)-I\right\|_{\mathcal{L}(V,W)}\|\F^\eta{v}_k\|_{V}\nonumber\\
&\quad \leq CT^\delta n^{-\delta} \|\F^\eta{v}_k\|_{V},\nonumber
\end{align}
and
\begin{align}
\|{v}_k\|_{V}&\leq C(T, C_1, C_2, \|f\|_{V}),\quad \|\F^\eta{v}_k\|_V\leq C(T, C_1, C_2, \|f\|_{V}). \nonumber
\end{align}
Let $w\defeq v-u$. Then the difference $w$ satisfies the equation
\begin{equation}
\left\{
\begin{aligned}
&\dfrac{\partial{w}}{\partial{t}}=A(t)w+X(v)-X(u)+R(t),\\
&w(0)=0.
\end{aligned}
\right.\non
\end{equation}
Using Duhamel's principle (see \cite{P83}), we know that
\begin{equation}
w(t)=\int_0^t\mathcal{U}(t,
\tau)\big[X(v(\tau))-X(u(\tau))+R(\tau)\big]\,d\tau.\non
\end{equation}
As a consequence, by assumptions \eqref{property of A}--\eqref{bound-X} and
the definition of $R(t)$, we obtain
\begin{align}
\|w(t)\|_V
& \leq\int_0^t\|\mathcal{U}(t,\tau)\|_{\LL(V)}\|X(v(\tau))-X(u(\tau))\|_Vd\tau+\int_0^t\|\mathcal{U}(t,\tau)\|_{\LL(W,V)}\|R(\tau)\|_{W}d\tau\non\\
& \leq\int_0^te^{C(t-\tau)}\|\mathcal{Y}(v(\tau),u(\tau))w(\tau)\|_Vd\tau+C\int_0^t(t-\tau)^{-\gamma}\|R(\tau)\|_{W}d\tau\non\\
& \leq\int_0^te^{C(t-\tau)}\|\mathcal{Y}(v(\tau),u(\tau))\|_{\LL(V)}\|w(\tau)\|_Vd\tau+C\int_0^t(t-\tau)^{-\gamma}\|X(\F^\eta{v}_k)-X(v)\|_{W}d\tau\non\\
&\quad+C\int_0^t(t-\tau)^{-\gamma}\Big\|\mathcal{U}\Big(\frac{k}{n}T+s,\frac{k}{n}T\Big)-I\Big\|_{\LL(V,W)}\|X(\F^\eta{v}_k)\|_{V}d\tau\non\\
& \leq
C\int_0^te^{C(t-\tau)}\|w(\tau)\|_Vd\tau+CT^\delta n^{-\delta}\int_0^t(t-\tau)^{-\gamma}\|\F^\eta{v}_k\|_Vd\tau\non\\
&\quad+CT^\delta n^{-\delta}\int_0^t(t-\tau)^{-\gamma}\|X(\F^\eta{v}_k)\|_Vd\tau\non\\
& \leq
C\int_0^te^{C(t-\tau)}\|w(\tau)\|_Vd\tau+ C(T, C_1,C_2,\|f\|_V)n^{-\delta}\int_0^t(t-\tau)^{-\gamma} d\tau\non\\
& \leq
C\int_0^te^{C(t-\tau)}\|w(\tau)\|_Vd\tau+C(T, C_1,C_2,\|f\|_V)t^{1-\gamma}n^{-\delta}.
\label{estimate-w}
\end{align}
Hence, the proof is complete by a direct application of Gronwall's Lemma to the inequality \eqref{estimate-w}.
\end{proof}

\section{Global existence for the limit system}
\setcounter{equation}{0}

In order to provide the global existence for the limit system as $\eps\to 0$, we first recall some technical result, concerning estimates in higher order space, in particular the  commutator estimates, that are nowadays standard (see for instance, \cite{Majdacompressible}).

\begin{lemma}
For any $f,g\in H^s(\TT)$ ($s\ge 2$, $s\in\N$) and any multi-index $\alpha$ with $|\alpha|\le s$, we have
\begin{align}
&\|\partial^\alpha_x (fg)\|_{L^2} \le c_s\left(\|f\|_{L^\infty}\|\nabla^s_x g\|_{L^2}+\|g\|_{L^\infty}\|\nabla^s_x f\|_{L^2}\right), \label{est:products}\\
&\|\partial^\alpha_x(fg)-f\partial^\alpha_x g\|_{L^2}\le c_s\left(\|\nabla_x f\|_{L^\infty} \|\nabla^{s-1}_x g\|_{L^2}+\|\nabla^s_x f\|_{L^2}\|g\|_{L^\infty}\right),\label{est:commutator}
\end{align}
for some positive constant $c_s$ independent of $f, g$.
\end{lemma}

We can now provide our global existence result:

\begin{proposition}\label{prop:globalexistlimitsyst}
Consider the following system for $\mathbf{v}:\TT\to \RR^3$ and $R:\TT\to\sS_0^{(3)}$:
\begin{align}
&\delt \mathbf{v}+ \mathbf{v}\cdot\nabla\mathbf{v}=-\nabla q,\label{nondim:v10app}\\
&\nabla\cdot \mathbf{v}=0,\label{nondim:dvapp}\\
&\delt R+\mathbf{v}\cdot\nabla R=(\xi D_{\mathbf{v}}+ \Omega_{\mathbf{v}})\big( R+\frac13 \Id
\big)+\big(R+\frac13 \Id \big)(\xi
D_{\mathbf{v}}- \Omega_{\mathbf{v}}) -2\xi\big(R+\frac13 \Id \big)\tr(R\nabla \mathbf{v})\non\\
&\qquad\qquad \qquad \quad -aR+b\big[R^2-\frac{1}{3}\tr(R^2)\Id\big]- cR\tr(R^2),
\label{nondimn:R10app}
\end{align}
with $\Omega_\mathbf{v}:=\displaystyle{\frac{\nabla \mathbf{v}-\nabla^T \mathbf{v}}{2}}$ and $D_\mathbf{v}:=\displaystyle{\frac{\nabla \mathbf{v}+\nabla^T \mathbf{v}}{2}}$.

If the initial data satisfy $(\mathbf{v}_0,R_0)\in H^5 (\TT;\RR^3)\times H^4(\TT; \sS_0^{(3)})$ then for any $T>0$, system \eqref{nondim:v10app}--\eqref{nondimn:R10app} admits a solution
$$(\mathbf{v},R)\in L^\infty\big(0,T;H^5 (\TT;\RR^3)\times H^4(\TT; \sS_0^{(3)})\big).$$

\end{proposition}

\begin{proof} For the $2D$ Euler system it is well known that if $\mathbf{v}_0\in H^5 (\TT;\RR^2)$ then there exists a global solution $\mathbf{v}$ in the same space, with values in $\mathbb{R}^2$. In our case we take values in $\mathbb{R}^3$, so for the third component indeed we have a transport equation with a Lipschitz velocity (provided by the first two components). Hence, by standard results we also have that the solution stays in the same space as the initial datum.

For the $R$-part, we just provide a priori estimates, leaving the approximation procedure to the interested reader.
Multiplying the equation \eqref{nondimn:R10app} by $R|R|^{p-2}$ ($p\geq 2$) and integrating by parts, we obtain:

\begin{align}
\frac{1}{p}\frac{d}{dt} \|R\|_{L^p}^p\le &\ \  C\|\nabla \mathbf{v}\|_{L^\infty}\|R\|_{L^p}^p+C\|\nabla \mathbf{v}\|_{L^p}\|R\|_{L^p}^{p-1}\non\\
&\ \  - \int_{\TT} \Big(2\xi |R|^p  \textrm{tr}(R\nabla \mathbf{v})+a|R|^p-b\textrm{tr}(R^3)|R|^{p-2}+c|R|^{p+2}\Big)\,dx.
\label{rel:estR0}
\end{align}
On the other hand, noting that $\textrm{tr}(R^3)$ can be expressed in terms of eigenvalues one can check that we have for any $\delta>0$ the estimate: $|\textrm{tr}(R^3)|\le \frac{3\delta}{8}|R|^4+\frac{3}{2\delta}|R|^2$, which implies:
\be
|b\textrm{tr}(R^3)|\le \frac{c}{4}|R|^4+\frac{9b^2}{4c}|R|^2. \non
\ee
Furthermore, we have
\be
2|\xi \textrm{tr}(R\nabla \mathbf{v})|\le \frac{c}{4}|R|^2+\frac{4\xi^2}{c}|\nabla \mathbf{v}|^2.\non
\ee
Using the above two estimates in \eqref{rel:estR0}, we get
\begin{align}
\frac{1}{p}\frac{d}{dt}\|R\|_{L^p}^p &\le C\|\nabla \mathbf{v}\|_{L^\infty}\|R\|_{L^p}^p+C\|\nabla \mathbf{v}\|_{L^p}\|R\|_{L^p}^{p-1}+C\|R\|_{L^p}^p\non\\
&\quad +C\|\nabla \mathbf{v}\|^2_{L^\infty}\|R\|_{L^p}^p-\frac{c}{2}\int_{\TT}|R|^{p+2}\,dx.\non
\end{align}
 Since $c>0$, we deduce that
\be
\frac{d}{dt}\|R\|_{L^p}\le C(\|\nabla \mathbf{v}\|_{L^\infty}^2+1)\|R\|_{L^p}+C\|\nabla \mathbf{v}\|_{L^p}.\non
\ee
Thus, out of the above nequality, using Gronwall's inequality, the a priori estimate on $\mathbf{v}$ and passing to the limit $p\to\infty$ we get
\be\label{est:Rinfty}
R\in L^\infty(0,T; L^\infty(\TT)).
\ee

We take now the scalar product in $H^4$ of  \eqref{nondimn:R10app} with $R$ and obtain
\begin{align}
\frac{1}{2}\frac{d}{dt} \|R\|_{H^4}^2=&-\underbrace{(\mathbf{v}\cdot \nabla R,R)_{H^4}}_{:=I_1}\non\\
&+\underbrace{\Big((\xi D_{\mathbf{v}}+ \Omega_{\mathbf{v}})\big( R+\frac13 \Id
\big)+\big(R+\frac13 \Id \big)(\xi
D_{\mathbf{v}}- \Omega_{\mathbf{v}}) -\frac{2}{3}\xi \tr(R\nabla \mathbf{v})\Id, R\Big)_{H^4}}_{:=I_2}\non\\
&-\underbrace{\big(2\xi R\tr(R\nabla \textbf{v}), R\big)_{H^4}}_{:=I_3}-a\|R\|_{H^4}^2+\underbrace{\left(bR^2-cR\tr(R^2),R\right)_{H^4}}_{:=I_4}\label{est:Rh4}
\end{align}
We estimate each term on the right-hand side separately.
\begin{align*}
|I_1|&\le C\underbrace{(\mathbf{v}\cdot\nabla R,R)_{L^2}}_{=0}+C|(\Delta^2 (\mathbf{v}\cdot\nabla R),\Delta^2 R)_{L^2}|\non\\
&=C|(\Delta^2 (\mathbf{v}\cdot\nabla R)-\mathbf{v}\cdot\nabla \Delta^2 R,\Delta^2 R)_{L^2}|\non\\
&\le C\|\mathbf{v}\|_{H^4}\|R\|_{H^4}^2,
\end{align*}
 where we used the estimate \eqref{est:commutator} in the last inequality.
Then for the second term we use repeatedly the estimate \eqref{est:products} to get
$$
|I_2|\le C\|\mathbf{v}\|_{H^5}\|R\|_{H^4}^2.
$$
Furthermore, for the third term, we have
\begin{align}
|I_3| &\le  C\|R^2\nabla \mathbf{v}\|_{H^4}\|R\|_{H^4}\non\\
&\le C\big(\|R\|_{L^\infty}\|R\nabla \mathbf{v}\|_{H^4}+\|R\nabla \mathbf{v}\|_{L^\infty}\|R\|_{H^4}\big) \|R\|_{H^4}\non\\
&\le C\Big[\|R\|_{L^\infty}\left(\|R\|_{L^\infty}\|\nabla \mathbf{v}\|_{H^4}+\|\nabla \mathbf{v}\|_{L^\infty}\|R\|_{H^4}\right)+\|R\|_{L^\infty}\|\nabla \mathbf{v}\|_{L^\infty}\|R\|_{H^4}\Big]\|R\|_{H^4}\non\\
&\le C\|R\|_{L^\infty}\|\mathbf{v}\|_{H^5}\|R\|_{H^4}^2.\non
\end{align}
A similar arguments lead to:
\be
|I_4|\le C\big(\|R\|_{L^\infty}+\|R\|_{L^\infty}^2\big)\|R\|_{H^4}^2.\non
\ee
Gathering the above estimates and using them in \eqref{est:Rh4}, we obtain
$$
\frac{d}{dt}\|R\|_{H^4}^2\le C\big(1+\|\mathbf{v}\|_{H^5}+\|R\|_{L^\infty}\|\mathbf{v}\|_{H^5}+\|R\|_{L^\infty}+\|R\|_{L^\infty}^2\big)\|R\|_{H^4}^2.\non
$$
Taking into account that $\|\mathbf{v}\|_{H^5}$ and $\|R\|_{L^\infty}$ are controlled a priori in $L^\infty(0,T)$, then using Gronwall's lemma, we deduce the desired a priori control of $R$ in $L^\infty(0,T; H^4)$.

The proof is complete.
\end{proof}

\section{Local ODE dynamics}
\setcounter{equation}{0}
Consider the ODE system
\begin{equation}\label{eq:gradODE}
\left\{
\begin{aligned}
&\frac{d}{dt}Q=-aQ+b\big(Q^2-\frac{1}{3}\tr(Q^2)\Id\big)-cQ\tr(Q^2),\\
&Q|_{t=0}=Q_0(x).
\end{aligned}
\right.
\end{equation}
Then we have
\begin{proposition}\label{prop:localODE}
Let  the coefficients of system \eqref{eq:gradODE} satisfy $a<0$, $b,c>0$. Denote by $$\mathcal{S}_*:=s_+\left\{\left(\mathbf{n}\otimes \mathbf{n}-\frac{1}{3}\Id\right),\,\mathbf{n}\in\mathbb{S}^2\right\}, \quad \text{where}\ \  s_+:=\frac{b+\sqrt{b^2-24ac}}{4c}$$
the set  of stationary points of the equation \eqref{eq:gradODE}.

There exists a neighbourhood $\mathcal{W}$ of $\mathcal{S}_*$, within the set $\sS_0^{(3)}$ of all $Q$-tensors, such that for any $Q_0\in \mathcal{W}$ the ODE system \eqref{eq:gradODE} starting from the initial datum $Q_0$ will evolve in the long time towards $\mathcal{T}[Q_0]:=s_+\displaystyle{\left(\mathbf{n}(Q_0)\otimes \mathbf{n}(Q_0)-\frac{1}{3}\Id\right)}$, with the function $\mathbf{n}:\mathcal{W}\to \mathbb{S}^2$ and $\mathcal{T}:\mathcal{W}\to\mathcal{S}_*$ continuous at all biaxial points in $\mathcal{W}$.
\end{proposition}
\begin{proof}
We start by recalling the argument from \cite{IXZ14}  that the dynamics of the ODE system \eqref{eq:gradODE}
 affects only the eigenvalues, but \emph{not} the eigenvectors.  Indeed, let us consider ``the system of eigenvalues":
\begin{equation}\label{eigensystem}
\left\{
\begin{aligned}
\frac{d\lambda_1}{dt}&=-\lambda_1\big[2c(\lambda_1^2+\lambda_2^2+\lambda_1\lambda_2)+a\big]+b\Big(\frac{\lambda_1^2}{3}
-\frac{2}{3}\lambda_2^2-\frac{2}{3}\lambda_1\lambda_2\Big),\\
\frac{d\lambda_2}{dt}&=-\lambda_2\big[2c(\lambda_1^2+\lambda_2^2+\lambda_1\lambda_2)+a\big]
+b\Big(\frac{\lambda_2^2}{3}-\frac{2}{3}\lambda_1^2-\frac{2}{3}\lambda_1\lambda_2\Big).
\end{aligned}
\right.
\end{equation}
 Using standard arguments it can be shown that the system \eqref{eigensystem} has a global-in-time solution $(\lambda_1, \lambda_2)$.

For an arbitrary initial datum $\tilde Q_0$, since
$\tilde Q_0$ is a symmetric matrix, there exists a matrix $R\in
O(3)$, such that
$
  R\tilde
  Q_0R^t=\diag (\tilde\lambda_1^0, \tilde\lambda_2^0,
   -\tilde\lambda_1^0-\tilde\lambda_2^0),
$
where $(\tilde \lambda_1^0,\tilde
\lambda_2^0,-\tilde\lambda_1^0-\tilde\lambda_2^0)$ are the
eigenvalues of $\tilde Q_0$. It can be shown  (see \cite{IXZ14} for details)
 that
 $$
  Q(t)=R^t\diag (\lambda_1(t) , \lambda_2(t) ,
  -\lambda_1(t)-\lambda_2(t))R.
$$
with $(\lambda_1(t),\lambda_2(t))$ being the solution of the ODE system \eqref{eigensystem}
 subject to the initial data $(\tilde\lambda_1^0,\tilde\lambda_2^0)$.

Therefore, for $Q_0$ close to $\mathcal{S}_*$, we have (taking into account that the eigenvalues are a continuous function of the matrix, see for instance \cite{nomizu}) that its eigenvalues are close to $-\frac{s_+}{3}$, $-\frac{s_+}{3}$ and $\frac{2s_+}{3}$, respectively. Thus, there exists a matrix $R\in
O(3)$, such that
$
  R
  Q_0R^t=\diag (\lambda_1^0, \lambda_2^0,
   -\lambda_1^0-\lambda_2^0),
$
where $(\lambda_1^0,
\lambda_2^0)-\frac{1}{3}(s_+,s_+)$ is small.

Denote $\hat X =(-\frac{s_+}{3},-\frac{s_+}{3})$.  A straightforward calculation shows that the Hessian matrix at $\hat X$ is  given by
\be
\left(
\begin{array}{ll} -a-12cs_+^2 & -2bs_+-6cs_+^2 \\   -2bs_+-6cs_+^2 &  -a-12cs_+^2\end{array}
\right).\non
\ee
Its eigenvalues take the following form
\be
\lambda_1=-a-2bs_+-18cs_+^2,\quad  \lambda_2=-a+2bs_+-6cs_+^2,\non
\ee
which can be checked to be negative under the assumption on $a,b$ and $c$. Thus, the standard ODE theory (see for instance,\cite{GHbook}) implies that for initial data in a neighbourhood of $\hat X$, we have that in the long time the solution will converge exponentially to $\hat X$. Let us denote by $\mathcal{U}_{\hat X}$ this neighbourhood.

We denote $$\mathcal{U}:=\{R\diag (\lambda,\mu,-\lambda-\mu)R^t,\ \  R\in O(3),\ \  (\lambda,\mu)\in \mathcal{U}_{\hat X}\},$$
which is a neighbourhood of $\mathcal{S}_*$.
We can see that the solution of the ODE system \eqref{eq:gradODE} starting from $Q_0\in \mathcal{U}$ will evolve in the long time to $\displaystyle{R(Q_0)\diag\Big(-\frac{s_+}{3},-\frac{s_+}{3},\frac{2s_+}{3}\Big)R(Q_0)^T}$ where $R(Q_0)$ is the matrix that diagonalizes $Q_0$ (namely, such that  $R(Q_0)
  Q_0R(Q_0)^T$ is diagonal). It can be shown that  $R(Q_0)$ is a matrix made of the eigenvectors of $Q_0$. Taking into account that a biaxial matrix has distinct eigenvalues and one can choose a smooth basis of eigenvectors near such a matrix (see for instance, \cite{nomizu}), we deduce that the function $R(Q_0)$ can be chosen in a continuous manner in a neighbourhood of any $Q_0$ biaxial. This proves the final claim of the Proposition.
\end{proof}

\end{appendix}

\section*{Acknowledgements}
 H. Wu is partially
supported by NNSFC grant No. 11631011. X. Xu is supported by the
start-up fund from the Department of Mathematics and Statistics at
Old Dominion University. A.Zarnescu is partially supported by a
Grant of the Romanian National Authority for Scientific Research and
Innovation, CNCS-UEFISCDI, project number PN-II-RU-TE-2014-4-0657;
by the Basque Government through the BERC 2014-2017 program; and by
the Spanish Ministry of Economy and Competitiveness MINECO: BCAM
Severo Ochoa accreditation SEV-2013-0323.



\begin{thebibliography}{99}
\itemsep=0pt

\bibitem{ADL14}
H. Abels, G. Dolzmann and Y.-N. Liu,
Well-posedness of a fully coupled Navier-Stokes/Q-tensor system with inhomogeneous boundary data,
SIAM J. Math. Anal., \textbf{46}, 3050--3077, 2014.

\bibitem{ADL15}
H. Abels, G. Dolzmann and Y.-N. Liu,
Strong solutions for the Beris-Edwards model for nematic liquid crystals with homogeneous Dirichlet boundary conditions,
Adv. Differential Equations, \textbf{21}, 109--152, 2016.

\bibitem{B84}
J. Ball,
Differentiability properties of symmetric and isotropic functions,
Duke Math. J.,  \textbf{51}, 699--728, 1984.

\bibitem{B12}
J. Ball,
Mathematics of liquid crystals,
Cambridge Centre for Analysis short course, 13--17, 2012.

\bibitem{BM10}
J. Ball and A. Majumdar,
Nematic liquid crystals: from Maier-Saupe to a continuum theory,
Mol. Cryst. Liq. Cryst., \textbf{525}, 1--11, 2010.

\bibitem{Virgaelectric}
R. Barberi, F. Ciuchi, G.E. Durand, M. Iovane, D. Sikharulidze, A.M. Sonnet and E.G. Virga,
Electric field induced order reconstruction in a nematic cell,
Euro. Phys. J. E: Soft Matter and Biological Physics, \textbf{13}(1), 61--71, 2004.

\bibitem{BC11}
A. B\'{a}tkai, P. Csom\'{o}s, B. Farkas and G. Nickel,
Operator splitting for non-autonomous evolution equations,
J. Func. Anal, \textbf{260}, 2163--2190, 2011.

\bibitem{BP16}
P. Bauman and D. Phillips,
Regularity and the behavior of eigenvalues for minimizers of a constrained Q-tensor energy for liquid crystals,
Calc. Var. Partial Differential Equations, \textbf{55}, 55--81, 2016,

\bibitem{BE94} A.-N. Beris and B.-J. Edwards,
\emph{Thermodynamics of Flowing Systems with Internal Microstructure},
Oxford Engineerin Science Series, \textbf{36}, Oxford university Press, Oxford, New York, 1994.

\bibitem{BV04}
S. Boyd and L. Vandenberghe,
\emph{Convex Optimization},
Cambridge University Press, Cambridge, 2004.

\bibitem{CRWX15}
C. Cavaterra, E. Rocca, H. Wu and X. Xu,
Global strong solutions of the full Navier--Stokes and Q-tensor system for nematic liquid crystal flows in two dimensions,
SIAM J. Math. Anal., \textbf{48}(2), 1368--1399, 2016.


\bibitem{DFRSS14}
M.~M. Dai, E. Feireisl, E. Rocca, G. Schimperna and M. Schonbek,
On asymptotic isotropy for a hydrodynamic model of liquid crystals,
Asymptot. Anal., \textbf{97}, 189--210, 2016.

\bibitem{D15}
F. De Anna,
A global 2D well-posedness result on the order tensor liquid crystal
theory, J. Differential Equations,  \textbf{262}(7), 3932--3979,  2017.

\bibitem{DAZ}
F. De Anna and A. Zarnescu,
Uniqueness of weak solutions of the full coupled Navier-Stokes and Q-tensor system in 2D,
Commun. Math. Sci., \textbf{14}, 2127--2178, 2016.


\bibitem{commutators}
C.L. Fefferman, D.S. McCormick, J.C.  Robinson and J.L. Rodrigo,
Higher order commutator estimates and local existence for the non-resistive MHD equations and related models,
J. Func. Anal., \textbf{267}(4), 1035--1056, 2014.

\bibitem{dG93}
P.G. de Gennes and  J. Prost,
\emph{The Physics of Liquid Crystals},
Oxford Science Publications, Oxford, 1993.

\bibitem{EV98}
L.C. Evans,
\emph{Partial Differential Equations},
Graduate Studies in Mathematics, \textbf{19}, American Mathematical
Society, Providence, RI, 1998.

\bibitem{EKT16}
L.C. Evans, O. Kneuss and H. Tran,
Partial regularity for minimizers of singular energy functionals, with application to liquid crystal models,
Trans. Amer. Math. Soc., \textbf{368}, 3389--3413, 2016.

\bibitem{FRSZ14}
E. Feireisl, E. Rocca, G. Schimperna and A. Zarnescu,
Evolution of non-isothermal Landau-de Gennes nematic liquid crystals flows with singular potential,
Commun. Math. Sci., \textbf{12}, 317--343, 2014.


\bibitem{GHbook}
J. Guckenheimer and P.J. Holmes.
\emph{Nonlinear Oscillations, Dynamical Systems, and Bifurcations of Vector Fields},
Vol. \textbf{42}, Springer Science \& Business Media, 2013.

\bibitem{GR14}
F. Guill\'en-Gonz\'alez and M.~A. Rodr\'iguez-Bellido,
Weak time regularity and uniqueness for a Q-tensor model,
SIAM J. Math. Anal., \textbf{46}, 3540--3567, 2014.

\bibitem{GR15}
F. Guill\'en-Gonz\'alez and M.~A. Rodr\'iguez-Bellido,
Weak solutions for an initial-boundary Q-tensor problem related to liquid crystals,
Nonlinear Anal., \textbf{112}, 84--104, 2015.

\bibitem{FRSZ15}
E. Feireisl, E. Rocca, G. Schimperna and A. Zarnescu,
Nonisothermal nematic liquid crystal flows with the Ball-Majumdar free energy,
Annali di Mat. Pura ed App., \textbf{194}(5), 1269--1299, 2015.



\bibitem{Hartmanbook}
P. Hartman,
\emph{Ordinary Differential Equations},
Reprint of the second edition, Birkh\"auser, Boston, Mass., 1982.

\bibitem{HSDbook}
M. W. Hirsch, S.  Smale, and R. L. Devaney,
\emph{Differential Equations, Dynamical Systems, and an Introduction to Chaos},
Academic press, 2012.

 \bibitem{kenig}
 A.D. Ionescu and C.E. Kenig,
 Local and global well-posedness of periodic KP-I equations,
``Mathematical Aspects of Nonlinear Dispersive Equations", Ann. Math. Stud., \textbf{163}, Princeton University Press, 181--212, 200

\bibitem{IXZ14}
G. Iyer, X. Xu and A. Zarnescu,
Dynamic cubic instability in a 2D Q-tensor model for liquid crystals,
Math. Models Methods Appl. Sci., \textbf{25}(8), 1477--1517, 2015.


\bibitem{kato}
T. Kato and G. Ponce,
Commutator estimates and the Euler and Navier--Stokes equations,
Commun. Pure Appl. Math., \textbf{41}(7), (1988), 891--907.


\bibitem{LiuCarme}
C. Liu  and M.C. Calderer,
Liquid crystal flow: dynamic and static configurations,
SIAM J. Appl. Math., \textbf{60}(6), 1925--1949, 2000.



\bibitem{NMreview}
N.J. Mottram and J.P. Newton,
Introduction to Q-tensor theory,
arXiv preprint, arXiv:1409.3542, 2014.

\bibitem{Majdacompressible}
A. Majda,
Compressible fluid flow and systems of conservation laws in several space variables,
Volume \textbf{53} of Applied Mathematical Sciences, Springer-Verlag, New York, 1984.

\bibitem{Majdabook}
A.J. Majda and A.L. Bertozzi,
\emph{Vorticity and Incompressible Flow},
Vol. \textbf{27}, Cambridge University Press. (2002).

\bibitem{M10}
A. Majumdar,
Equilibrium order parameters of nematic liquid crystals in the Landau-de Gennes theory,
European J. Appl. Math., \textbf{21}, 181--203, 2010.

\bibitem{MZ10}
A. Majumdar and A. Zarnescu,
Landau-De Gennes theory of nematic liquid crystals: the Oseen-Frank limit and beyond,
Arch. Rational Mech. Anal., \textbf{196}, 227--280, 2010.

\bibitem{nomizu}
K. Nomizu,
Characteristic roots and vectors of a diifferentiable family of symmetric matrices,
Linear and Multilinear Algebra, \textbf{1}(2), 159--162, 1973.


\bibitem{PZ11}
M. Paicu and A. Zarnescu,
Global existence and regularity for the full coupled Navier-Stokes and Q-tensor system,
SIAM J. Math. Anal., \textbf{43}, 2009--2049, 2011.

\bibitem{PZ12}
M. Paicu and A. Zarnescu,
Energy dissipation and regularity for a coupled Navier-Stokes and Q-tensor system,
Arch. Ration. Mech. Anal., \textbf{203}, 45--67, 2012.

\bibitem{MkadGart}
S. Mkaddem, and E. C. Gartland Jr.,
Fine structure of defects in radial nematic droplets,
Physical Review E, \textbf{62}(5) (2000): 6694.

\bibitem{P83}
A. Pazy,
\emph{Semigroups of Linear Operators and Applications to Partial Differential Equations},
Applied Mathematical Sciences, \textbf{44}, Springer-Verlag, New York, 1983.


\bibitem{T97}
M.~E. Taylor,
\emph{Partial Differential Equations. III. Nonlinear Equations},
Applied Mathematical Sciences, \textbf{117}, Springer-Verlag, New York, 1997.

\bibitem{W12} M. Wilkinson,
Strict physicality of global weak solutions of a Navier-Stokes Q-tensor system with singular potential,
Arch. Ration. Mech. Anal., \textbf{218}, 487--526, 2015.


\bibitem{Z04}
S.~Zheng,
\emph{Nonlinear Evolution Equations},
Pitman series
Monographs and Survey in Pure and Applied Mathematics, \textbf{133}, Chapman \& Hall/CRC, Boca Raton, Florida, 2004.

\end{thebibliography}
\end{document}